\RequirePackage{fix-cm}
\documentclass[smallextended]{svjour3}       
\providecommand{\smartqed}{}
\usepackage{springerpaper}

\usepackage{tikz}
\usepackage{pgfplots}
\usepackage{tikzscale}
\usetikzlibrary{external}
\def\figurespath{figs}
\graphicspath{{\figurespath/}}

\newcommand*{\includetikzgraphics}[2][]{%
  \includegraphics[#1]{#2.pdf}%
}

\makeatletter
\renewcommand{\todo}[2][]{\tikzexternaldisable\@todo[#1]{#2}\tikzexternalenable}
\newcommand*{\defeq}{\mathrel{\vcenter{\baselineskip0.5ex \lineskiplimit0pt
                     \hbox{\scriptsize.}\hbox{\scriptsize.}}}%
                     =}
\makeatother

\pgfplotsset{
  colormap={myhot}{
    rgb=(0,0,0)      
    rgb=(0.3,0,0.3)  
    rgb=(0.7,0,0)    
    rgb=(1,0.8,0)    
    rgb=(1,1,1)      
  }
}

\pgfplotsset{compat=newest}
\pgfplotsset{every axis legend/.append style={%
cells={anchor=west}}
}
\usepgfplotslibrary{groupplots}
\usepgfplotslibrary{colormaps}

\usetikzlibrary{arrows}
\tikzset{>=stealth'}

\spnewtheorem{assumption}{Assumption}{\it}{\rm}
\usepackage{cleveref}
\crefname{assumption}{assumption}{assumptions}
\Crefname{assumption}{Assumption}{Assumptions}
\Crefname{algorithm}{Algorithm}{Algorithms} 
\newcommand{\cahiernumber}{73}  
\newcommand{\nf}{\nabla f}
\newcommand{\kf}{\kappa_f}
\newcommand{\knf}{\kappa_{\nabla}}
\newcommand{\hf}{\widehat{f}}
\newcommand{\hnf}{\widehat{\nabla} f}
\newcommand{\hxicp}{\widehat{\xi}_{\mathrm{cp}}}
\newcommand{\hxikcp}{\widehat{\xi}_{k, \mathrm{cp}}}

\newcommand{\hskcpj}{\widehat{s}_{k, \mathrm{cp}}^{(j)}}
\newcommand{\xicp}{\xi_{\mathrm{cp}}}
\newcommand{\xikcp}{\xi_{k, \mathrm{cp}}}
\newcommand{\hsk}{s_k}
\newcommand{\skcp}{s_{k, \mathrm{cp}}}
\newcommand{\skecp}{s_{k_\epsilon, \mathrm{cp}}}
\newcommand{\hskcp}{\widehat{s}_{k, \mathrm{cp}}}
\newcommand{\hskecp}{\widehat{s}_{k_\epsilon, \mathrm{cp}}}
\newcommand{\varphicp}{\varphi_{\mathrm{cp}}}
\newcommand{\mcp}{m_{\mathrm{cp}}}
\newcommand{\scp}{s_{\mathrm{cp}}}
\newcommand{\hscp}{\widehat{s}_{\mathrm{cp}}}
\newcommand{\dom}{\mathrm{dom}}
\newcommand{\B}{\mathds{B}}
\journalname{Computational Optimization and Applications}
\def\thistitle{An Inexact Modified Quasi-Newton Method for Nonsmooth Regularized Optimization}
\def\authorone{Nathan Allaire}
\def\authortwo{Sébastien Le Digabel}
\def\authorthree{Dominique Orban}
\hypersetup{
  pdftitle={\thistitle},
  pdfauthor={\authorone{} and \authortwo{} and \authorthree{}},
  pdfsubject={Nonsmooth Regularized Optimization},
  pdfkeywords={nonsmooth optimization,
    nonconvex optimization,
    modified quasi-Newton method,
    proximal quasi-Newton method,
    regularized optimization,
    composite optimization,
    proximal gradient method,
    inexact proximal operator,
    inexact evaluations}
}
\begin{document}

\title{%
  \thistitle
  \thanks{%
    This research was supported by NSERC Discovery Grants RGPIN-2024-05086, RGPIN-2020-06535, and RGPIN-2025-06911.
  }
}

\titlerunning{Inexact Modified QN for Regularized Optimization}        

\author{\authorone{} \and
  \authortwo{} \and
  \authorthree{}}


\institute{%
  \authorone{} \at
  GERAD and Department of Mathematics and Industrial Engineering, Polytechnique Montr\'eal.
  \email{nathan.allaire@polymtl.ca}           
  \and
  \authortwo{} \at
  GERAD and Department of Mathematics and Industrial Engineering, Polytechnique Montr\'eal.
  \email{sebastien.le-digabel@polymtl.ca}           
  \and
  \authorthree{} \at
  GERAD and Department of Mathematics and Industrial Engineering, Polytechnique Montr\'eal.
  \email{dominique.orban@gerad.ca}           
}

\date{Received: date / Accepted: date}

\pagestyle{myheadings}

\maketitle
\thispagestyle{mytitlepage}

\begin{abstract}
  We introduce method iR2N, a modified proximal quasi-Newton method for minimizing the sum of a \(\mathcal{C}^1\) function \(f\) and a lower semi-continuous prox-bounded \(h\) that permits inexact evaluations of \(f\), \(\nabla f\) and of the relevant proximal operators.
  Both \(f\) and \(h\) may be nonconvex.
  In applications where the proximal operator of \(h\) is not known analytically but can be evaluated via an iterative procedure that can be stopped early, or where the accuracy on \(f\) and \(\nabla f\) can be controlled, iR2N can save significant computational effort and time.
  At each iteration, iR2N computes a step by approximately minimizing the sum of a quadratic model of \(f\), a model of \(h\), and an adaptive quadratic regularization term that drives global convergence.
  In our implementation, the step is computed using a variant of the proximal-gradient method that also allows inexact evaluations of the smooth model, its gradient, and proximal operators.
  We assume that it is possible to interrupt the iterative process used to evaluate proximal operators when the norm of the current iterate is larger than a fraction of that of the minimum-norm optimal step, a weaker condition than others in the literature.
  Under standard assumptions on the accuracy of \(f\) and \(\nabla f\), we establish global convergence in the sense that a first-order stationarity measure converges to zero and a worst-case evaluation complexity in \(O(\epsilon^{-2})\) to bring said measure below \(\epsilon > 0\).
  Thus, inexact evaluations and proximal operators do not deteriorate asymptotic complexity compared to methods that use exact evaluations.
  We illustrate the performance of our implementation on problems with \(\ell_p\)-norm, \(\ell_p\) total-variation and the indicator of the nonconvex pseudo \(p\)-norm ball as regularizers.
  On each example, we show how to construct an effective stopping condition for the iterative method used to evaluate the proximal operator that ensures satisfaction of our inexactness assumption.
  Our results show that iR2N offers great flexibility when exact evaluations are costly or unavailable, and highlight how controlled inexactness can reduce computational effort effectively and significantly.
\end{abstract}

%


\section{Introduction}%

We consider the problem class
\begin{equation}%
  \label{eq:problem-adressed}
  \minimize{x \in \R^n} f(x) + h(x),
\end{equation}
where \(f : \R^n \to \R\) is continuously differentiable, \(h : \R^n \to \R \cup \{+\infty\}\) is proper, lower semi-continuous (lsc), and both may be nonconvex.
In practice, \(h\), called the \emph{regularizer}, is designed to promote desirable properties in solutions, such as sparsity.
We develop method iR2N, a variant of the modified proximal quasi-Newton algorithm R2N of \citet{diouane-habiboullah-orban-2024} that allows for inexact evaluations of \(f\) and \(\nabla f\), as well as of the relevant proximal operators.
Among other applications, evaluations of \(f\) and \(\nabla f\) are inexact when they result from the discretization of a differential or integral operator \citep{baraldi-kouri-2023}, from the sampling of a sum of a large number of terms, as in machine learning applications \citep{robbins-monro-1951}, or from using multiple floating-point systems \citep{monnet-orban-2025}.
Like R2N, iR2N computes a step at each iteration by approximately minimizing the sum of a quadratic model of \(f\), a model of \(h\), and an adaptive quadratic regularization term.
The subproblem is solved with method iR2, which is to method R2 of \citet{aravkin-baraldi-orban-2022} as iR2N is to R2N, i.e., proximal operators are evaluated inexactly.
Method R2 may be viewed as a variant of the standard proximal-gradient method with adaptive step length, and is a special case of R2N.
We consider settings where proximal operators do not have a closed-form expression, and one must thus rely on inexact evaluations.
Specifically, we focus on scenarios where proximal operators can be evaluated by running a convergent algorithm that can be terminated early with appropriate guarantees detailed below.
Special cases that fit our assumptions include choices of convex and nonconvex \(h\), including the \(\ell_p\)-norm total variation (TV), \(\ell_p\)-norm regularizer and the indicator of the nonconvex \(\ell_p\)-pseudo norm ball with \(0 < p < 1\).
Method iR2N reduces to R2N when \(f\), \(\nabla f\) and proximal operators are evaluated exactly.
We establish global convergence of iR2N under standard assumptions on the inexactness of \(f\) and \(\nabla f\), and provided the inexact proximal operator yields a step whose norm is at least a fraction of the norm of an optimal step.
We also establish that worst-case evaluation complexity of iR2N is of the same order as that of R2N.
Thus, inexact evaluations do not degrade worst-case complexity.
Our remaining assumptions are standard.
To emphasize our assumptions on inexact evaluations, we simplify those assumptions of \citep{diouane-habiboullah-orban-2024} that would complicate the analysis.
In particular, we assume that \(\nabla f\) is Lipschitz continuous, but its Lipschitz constant need not be known nor approximated.
However, it should be clear that iR2N remains convergent under the more general assumptions of \citep{diouane-habiboullah-orban-2024} with its worst-case complexity affected accordingly.
It should also be clear that minor alterations of our approach would establish that the proximal quasi-Newton trust-region algorithm of \citet{aravkin-baraldi-orban-2021,aravkin-baraldi-leconte-orban-2021} remains convergent under inexact evaluations and its asymptotic worst-case complexity is unchanged.
Such minor alterations would also establish convergence and complexity of Levenberg-Marquardt variants in the vein of \citep{aravkin-baraldi-orban-2024} that are useful when \(f\) is a least-squares residual.

We report computational experience with the proximal operator of the \(\ell_p\) norm, the total variation in \(\ell_p\) norm, and the indicator of the nonconvex \(\ell_p\)-pseudo norm ball.
Each of those proximal operators must be evaluated via an iterative procedure.
For each, we devise a stopping condition that ensures satisfaction of our assumption on inexact proximal operator evaluations.
Our results show that iR2N offers great flexibility in settings where exact evaluations are costly or unavailable, and highlight how controlled inexactness can be exploited to reduce computational effort effectively and significantly.
We provide an efficient Julia implementation of iR2N as part of the open-source package \texttt{RegularizedOptimization.jl} \citep{baraldi-leconte-orban-2024}.

\subsection*{Related Research}

Most numerical methods for~\eqref{eq:problem-adressed} require the evaluation of one or more proximal operators \citep{moreau-1965} at each iteration.
The proximal operator of \(h\) with step size \(\nu > 0\) at \(q \in \R^n\) is
\begin{equation}%
  \label{eq:def-prox}
  \prox{\nu h}(q) \defeq \argmin{u \in \R^n} \tfrac{1}{2} \|u - q\|^2 + \nu h(u) \subseteq \R^n.
\end{equation}
For given \(h\) and \(q\),~\eqref{eq:def-prox} can be empty, a singleton or contain multiple elements, one of which must be identified.
\citet{beck-2017} and \citet{chierchia-2020} summarize the closed-form of~\eqref{eq:def-prox} for a large number of choices of \(h\) relevant in applications.
The standard proximal-gradient method \citep{fukushima-mine-1981} along with most proximal methods in the literature assume that obtaining an element of~\eqref{eq:def-prox} \emph{exactly} is possible.

For certain choices of \(h\), it is necessary to apply an iterative method to appoximate an element of~\eqref{eq:def-prox}, e.g., the total variation (TV) with \(\ell_p\) regularization \(h(x) = \|Dx\|_p\), where \(p \geq 1\) and \(D\) is the upper bidiagonal finite-difference operator with a diagonal of negative ones and a superdiagonal of ones.
Finding an element in~\eqref{eq:def-prox} for the TV-\(\ell_p\) can be achieved via the taut-string method \citep{barbero-sra-2018} or the fast TV denoising method \citep{condat-2013}.
As in other methods in the literature for various choices of convex \(h\) \citep{barre-taylor-bach-2023,dai-robinson-2022,gu-wang-huo-huang-2018}, the latter monitor the duality gap between a convex problem and its dual.
Those algorithms have guaranteed convergence properties and can be terminated early, i.e., short of optimality.
In the above, the evaluation of~\eqref{eq:def-prox} is inexact in the sense that a convergent process to identify a global minimizer is applied and can be stopped short of optimality according to an optimality criterion.

A somewhat more complicated scenario is the algorithm described by \citet{yang-wang-wang-2022} for the case where \(h\) is the indicator of the ``ball'' in pseudo-norm \(\ell_p\) with \(p \in (0, 1)\).
The evaluation of the proximal operator requires solving a nonconvex problem to global optimality in that case, and their algorithm is not guaranteed to always succeed.
We return to this problem in \Cref{sec:implementation}.

Other concepts of inexactness of the proximal operator appear in the literature.
For convex \(h\), \citet{rockafellar-1976} requires that an approximate solution of~\eqref{eq:def-prox} be a certain distance from the optimal set.
Still for convex \(h\), \citet{barre-taylor-bach-2023} unveil multiple ways to define inexactness by finding a primal-dual point in a certain relaxed subdifferential.
\citet{salzo-villa-2012} define three approximations: they compute \(z\) such that either \emph{(i)} \(\|z - \prox{\nu h}(q)\| \le \epsilon\), \emph{(ii)} \(\nu^{-1} (q-z)\) lies in a relaxation of the subdifferential of \(h\) at \(z\), or \emph{(iii)} \(z \in \prox{\nu h}(q + e)\) with \(\|e\| \leq \epsilon\) for some \(\epsilon > 0\).
\citet{chen-yu-huang-2025} extend proximal inexactness by introducing the concept of \((\gamma,\delta,\varepsilon)\)-proximal-gradient stationary point (PGSP) for convex \(h\) based on the Goldstein subdifferential.
The PGSP generalizes the three concepts of \citep{salzo-villa-2012} by jointly relaxing spatial and functional exactness and directly quantifying the first-order residual, thus also encompassing Rockafellar’s \citep{rockafellar-1976} and relaxed subgradient formulations within a unified framework.
For nonconvex \(h\), \citet{gu-wang-huo-huang-2018} say that an element is an inexact solution of~\eqref{eq:def-prox} if its objective value is within \(\epsilon\) of its optimal value.

To cope with inexact evaluations of the proximal operator, classical schemes must be revised to preserve convergence guarantees.
The seminal inexact proximal-point algorithm (iPPA) of \citet{rockafellar-1976} allows summably controlled errors in the resolvant computation of a maximal monotone operator and still ensures global convergence with linear/superlinear behavior under suitable parameter growth.
Building on the accelerated estimate-sequence framework, \citet{salzo-villa-2012} establish that the accelerated iPPA retains \(O(1/k)\) decay under inexactness of type (i) above, and optimal \(O(1/k^2)\) decay under inexactness of type (ii).
\citet{schmidt-roux-bach-2011} establish an \(O(1/k)\) rate for proximal-gradient and an \(O(1/k^2)\) rate for an accelerated variant under inexactness similar to (iii) above.
Extensions include inertial, variable-metric forward–backward schemes with relative inner accuracy and uniform symmetric positive definite metrics \citep{bonettini-rebegoldi-ruggiero-2018}; nonconvex inexact (accelerate) proximal gradient with guarantees matching the exact counterparts under calibrated error schedules \citep{gu-wang-huo-huang-2018}; adaptive, implementable stopping rules that preserve \(O(\epsilon^{-2})\) iteration complexity and enable support identification \citep{dai-robinson-2022}; and accelerated proximal gradient under relative error criteria that maintain an \(O(1/k^2)\) rate \citep{bellocruz-goncalves-krislock-2020}.
For nonconvex problems, the sequence generated by an inexact proximal-gradient (or splitting) method can still be shown to converge to a first-order critical point under an assumption of type (iii) above on the approximation errors \citep{sra-2012}.
Finally, for weakly convex functions, recent results establish global convergence for inexact proximal algorithms under inexactness of type (i) above, allowing controlled inexactness in the proximal steps while maintaining convergence \citep{khahn-mordukhovich-phat-tran-2025}.

\subsection*{Notation}

The Euclidean norm is \(\|\cdot\|\).
When required, other norms are denoted with different symbols.
We use \(f\), \(h\), \(m\), \(\phi\), \(\varphi\), \(\xi\) and \(\psi\) for functions.
Other lowercase Latin letters denote vectors in \(\R^n\).
Exceptions are \(p\) and \(q\), which are standard to denote a pair of dual \(\ell_p\) and \(\ell_q\) norms, and \(r\), which denotes a radius.
Uppercase \(A\) and \(B\) are matrices, \(L\) is a Lipschitz constant, and \(O\) is used for the Landau notation.
Lowercase Greek letters denote scalars.
Calligraphic letters denote sets.

\section{Background}

\subsection{Variational Analysis Concepts}

We say that \(h: \R^n \to \R \cup \{+\infty\}\) is proper if \(h(x) < +\infty\) for at least one \(x \in \R^n\) and lower semi-continuous (lsc) at \(\bar{x} \in \R^n\) if \(\liminf_{x \to \bar{x}} h(x) = h(\bar{x})\).
It is lsc if it is lsc at all \(\bar{x} \in \R^n\).
We say that \(h\) is prox-bounded at \(x\) if there is \(\lambda > 0\) such that \(w \mapsto h(w) + \tfrac{1}{2} \lambda^{-1} \|w - x\|^2\) is bounded below \citep[Definition 1.23]{rockafellar-wets-2009}.
The threshold of prox-boundedness of \(h\) at \(x\) is the supremum of all such \(\lambda\) at \(x\), and is denoted \(\lambda_x\).
We say that \(h\) is \emph{uniformly prox-bounded} if there is \(\lambda \in \R_+ \cup \{+\infty\}\) such that \(\lambda_x \geq \lambda\) for all \(x \in \R^n\).

%
For \(\phi : \R^n \to \R \cup \{+ \infty\}\) and \(\bar{x} \in \dom(\phi)\), the Fréchet subdifferential of \(\phi\) at \(\bar{x}\) is 
\[
  \widehat{\partial} \phi(\bar{x}) \defeq
  \left\{
  v \in \R^n
  \vphantom{\liminf_{x \to \bar{x}} \frac{\phi(x) - \phi(\bar{x}) - v^T (x - \bar{x})}{\|x - \bar{x}\|}}
  \right.
  \,
  \left|
  \,
  \liminf_{x \to \bar{x}} \frac{\phi(x) - \phi(\bar{x}) - v^T (x - \bar{x})}{\|x - \bar{x}\|} \geq 0
  \right\}.
\]
The limiting subdifferential \(\partial \phi(\bar{x})\) of \(\phi\) at \(\bar{x}\) is the set of elements \(v \in \R^n\) such that there exists a sequence \(\{x_k\} \to \bar{x}\) with \(\{\phi(x_k)\} \to \phi(\bar{x})\), and there exists \(v_k \in \widehat{\partial} \phi(x_k)\) for all \(k\) such that \(\{v_k\} \to v\).
It always holds that \(\widehat{\partial} \phi(\bar{x}) \subseteq \partial \phi(\bar{x})\).

If \(\phi\) is proper, we say that \(\bar{x}\) is stationary for \(\phi\), or for the problem of minimizing \(\phi\), if \(0 \in \widehat{\partial} \phi(\bar{x})\).
If \(\phi\) is proper and has a local minimum at \(\bar{x}\), then \(\bar{x}\) is stationary for \(\phi\).
In the special case where \(\phi = f + h\) with \(f\) continuously differentiable and \(h\) proper, then \(\partial \phi(x) = \nabla f(x) + \partial h(x)\)~\citep[Theorem 10.1]{rockafellar-wets-2009}.
We say that \(f: \R^n \to \R\) has Lipschitz-continuous gradient with Lipschitz constant \(L \geq 0\) if for all \(x\) and \(s \in \R^n\),
\begin{equation}%
  \label{eq:lipschitz-gradient}
  |f(x + s) - f(x) - \nabla f(x)^T s| \leq \tfrac{1}{2} L \|s\|^2.
\end{equation}

\subsection{Models}

In this work, we focus on three sources of inexactness: the objective, its gradient and the proximal operator evaluations.
We denote \(\hf\) and \(\hnf\) the inexact counterparts of \(f\) and \(\nf\).
At each iteration, R2N computes a step \(\scp\) defined below that serves to define a stationarity measure and that results from a proximal operator evaluation.
Accordingly, in iR2N, we denote its inexact counterpart \(\hscp\).
We follow \citep{aravkin-baraldi-orban-2022,aravkin-baraldi-orban-2024,diouane-habiboullah-orban-2024} and structure the iterations of an algorithm around two sets of models, but, since the only information we have access to is inexact, those are based on \(\hf\) and \(\hnf\).

For \(\nu > 0\) and \(x \in \R^n\), the first-order models
\begin{align}%
  \label{eq:phi-first-order}
  \varphicp(s; x)      & \defeq \hf(x) + \hnf(x)^T s                                         \\
  %
  \psi(s; x)           & \phantom{:}\approx h(x + s)                                         \\
  \label{eq:model-first-order}
  \mcp(s; x, \nu^{-1}) & \defeq \varphicp(s; x) + \tfrac{1}{2} \nu^{-1} \|s\|^2 + \psi(s; x)
\end{align}
serve to generalize the concept of Cauchy point, hence the subscript ``cp'', where we use the symbol ``\(\approx\)'' to mean that the left-hand side is an approximation of the right-hand side.
We will be more specific in \Cref{ass:psi-accuracy} below.
The dual role of models~\eqref{eq:phi-first-order}--\eqref{eq:model-first-order} is to define a threshold for sufficient decrease at each iteration, and to define a measure of approximate stationarity.

For \(\sigma > 0\), \(x \in \R^n\) and \(B(x) = B(x)^T \in \R^{n \times n}\), the second-order models
\begin{align}%
  \label{eq:phi-second-order}
  \varphi(s; x)   & \defeq \hf(x) + \hnf(x)^T s + \tfrac{1}{2} s^T B(x) s            \\
  \label{eq:model-second-order}
  m(s; x, \sigma) & \defeq \varphi(s; x) + \tfrac{1}{2} \sigma \|s\|^2 + \psi(s; x),
\end{align}
are used to compute a step.
Because \(\varphicp(\cdot; x)\) is linear and \(\varphi(\cdot; x)\) is quadratic for fixed \(x\), both have globally Lipschitz-continuous gradient.

We follow \citep{aravkin-baraldi-orban-2022,aravkin-baraldi-orban-2024,diouane-habiboullah-orban-2024} and require that all models that we consider satisfy the following assumption.

\begin{assumption}
  \label{ass:psi-prox-bounded}
  For all \(x \in \R^n\), \(\psi(\cdot; x)\) is proper, lsc and uniformly prox-bounded.
  In addition, \(\psi(0;x) = h(x)\) and \(\partial \psi (0;x) \subseteq \partial h(x)\).
\end{assumption}

\subsection{The Proximal-Gradient Method}%

The direct generalization of the gradient method to nonsmooth regularized optimization is the proximal-gradient method \citep{fukushima-mine-1981}.
For~\eqref{eq:problem-adressed}, the inexact proximal-gradient iteration can be written
\begin{align}
  x_{k+1} & =   x_k + \skcp
  \\
  \skcp   & \in \argmin{s} \tfrac{1}{2} \nu_k^{-1} \|s + \nu_k \hnf(x_k)\|^2 + \psi(s; x_k) \notag
  \\
          & =   \argmin{s} \hnf(x_k)^T s + \tfrac{1}{2} \nu_k^{-1} \|s\|^2 + \psi(s; x_k)
  \label{eq:proximal-problem}
  \\
          & =   \argmin{s} \mcp(s; x_k, \nu_k^{-1}), \notag
\end{align}
where \(\nu_k > 0\) is an appropriate step length, though it is typically used with \(\psi(s; x_k) \defeq h(x_k + s)\).
We call \(\skcp\) a Cauchy step.
It turns out that \(\skcp\) exists provided \(\nu_k\) is sufficiently small.

\begin{proposition}[{\protect \citealp[Theorem~\(1.25\)]{rockafellar-wets-2009}}]%
  \label{prop:skcp}
  Let \(\varphicp(\cdot; x)\) be as in~\eqref{eq:phi-first-order}, and \(\psi(\cdot; x)\) be proper, lsc, prox-bounded with threshold \(\lambda_x > 0\) and such that \(\psi(0; x) = h(x)\).
  For any \(0 < \nu < \lambda_x\), the set \(\argmin{s} \mcp(s; x, \nu^{-1})\) is nonempty and compact.
\end{proposition}

We denote \(\scp\) an element of \(\argmin{s} \mcp(s; x, \nu^{-1})\) when one exists.
When \(\scp\) is well defined, the quantity
\begin{equation}%
  \label{eq:xicp}
  \begin{aligned}
    \xicp(\scp, x, \nu^{-1}) & \defeq (\varphicp + \psi)(0; x) - (\varphicp + \psi)(\scp; x)
    \\
                             & \phantom{:}= (\hf + h)(x) - (\varphicp + \psi)(\scp; x)
  \end{aligned}
\end{equation}
is central to the algorithm and the analysis, as it is in \citep{aravkin-baraldi-orban-2022,aravkin-baraldi-orban-2024,diouane-habiboullah-orban-2024}, where it plays the dual role of defining Cauchy decrease and serving as stationarity measure.
Indeed, under standard assumptions, \(x\) is stationary for~\eqref{eq:problem-adressed} if \(\xicp(\scp; x, \nu^{-1}) = 0\) \citep[Lemma~\(3.5\)]{diouane-habiboullah-orban-2024}.
We diverge slightly from those references and, for reasons that become clear later, note that \(\nu^{-1} \|\scp\|\) can equally be used as stationarity measure.

\begin{proposition}%
  \label{prop:stationarity}
  Let \(x \in \R^n\) and \(\psi(\cdot; x)\)  be proper, lsc, prox-bounded with threshold \(\lambda_x > 0\) and such that \(\partial \psi(0; x) \subseteq \partial h(x)\).
  Let \(0 < \nu < \lambda_x\) and \(\scp \in \argmin{s} \mcp(s; x, \nu^{-1})\).
  If \(\scp = 0\), then \(0 \in \hnf(x) + \partial h(x)\).
  If, in addition, \(\hnf(x) = \nabla f(x)\), then \(x\) is stationary for~\eqref{eq:problem-adressed}.
\end{proposition}

\begin{proof}
  If \(\scp = 0\), then \(\xicp(\scp; x, \nu^{-1}) = 0\) by~\eqref{eq:xicp}.
  The rest of the proof is identical to those of \citep[Lemma~\(3.5\)]{diouane-habiboullah-orban-2024} and \citep[Lemma~\(3.3\)]{aravkin-baraldi-orban-2024}.
  \qed
\end{proof}

In the special case \(h = 0\), i.e., smooth optimization, \(\scp = -\nu \nabla f(x)\).
Thus, we normalize and use \(\nu^{-1} \|\scp\|\) as stationarity measure.

The identification of an \(\scp\), when one exists, coincides with the identification of an element in the image of a proximal operator~\eqref{eq:def-prox}, i.e., \(\scp \in \prox{\nu \psi(\cdot; x)}(-\nu \hnf(x))\).
It is the computation of an element in such a set that represents the main computational challenge in problems for which the set is not known analytically, so that one must resort to an iterative numerical method.
In that case, the \(\scp\) computed is inexact, and we refer to this situation as an inexact evaluation of the proximal operator.

The following result is hidden inside the proof of \citep[Lemma~\(2\)]{bolte-sabach-teboulle-2014}.

\begin{proposition}%
  \label{prop:sufficient-decrease}
  Let \(f\) have Lipschitz-continuous gradient with Lipschitz constant \(L \geq 0\) and let \(h\) be proper, lsc and prox-bounded at \(x \in \R^n\) with threshold \(\lambda_x > 0\).
  Let \(0 < \nu < \min(1/L, \, \lambda_x)\), and let \(s \in \R^n\) be such that
  \begin{equation}%
    \label{eq:suff-decrease}
    f(x) + \nabla f(x)^T s + \tfrac{1}{2} \nu^{-1} \|s\|^2 + h(x + s) \leq (f + h)(x).
  \end{equation}
  Then,
  \begin{equation}%
    \label{eq:xi-inequality}%
    (f + h)(x)
    - (f + h)(x + s) \geq \tfrac{1}{2} (\nu^{-1} - L) \|s\|^2.
  \end{equation}
\end{proposition}

\begin{proof}
  We inject
  \(
  f(x) + \nabla f(x)^T s \geq f(x + s) - \tfrac{1}{2} L \|s\|^2,
  \)
  which follows from~\eqref{eq:lipschitz-gradient}, into~\eqref{eq:suff-decrease} and obtain~\eqref{eq:xi-inequality}.
  \qed
\end{proof}

\Cref{prop:sufficient-decrease} applied to \(\varphicp(\cdot; x)\), \(\psi(\cdot; x)\) and \(\scp \in \prox{\nu \psi(\cdot; x)} (-\nu \hnf(x))\), yields
\begin{equation}%
  \label{eq:xikcp-inequality}
  \xicp(\scp; x, \nu^{-1}) \geq \tfrac{1}{2} \nu^{-1} \|\scp\|^2,
\end{equation}
because the Lipschitz constant of \(\nabla \varphicp(\cdot; x)\) is zero.

By contrast, we denote an approximate Cauchy step resulting from an \emph{inexact} minimization of~\eqref{eq:model-first-order} as \(\hscp\).
We will be more specific about the meaning of inexactness in that context in \Cref{ass:inexact-decrease}.
Accordingly, we define
\begin{equation*}%
  \hxicp(\hscp; x, \nu^{-1}) \defeq (\varphicp + \psi)(0; x) - (\varphicp + \psi)(\hscp; x).
\end{equation*}
\Cref{prop:sufficient-decrease} states that~\eqref{eq:xi-inequality} also holds for any \(s\) that produces simple decrease in~\eqref{eq:model-first-order}; \(s\) need not be an exact minimizer.
Thus, if we apply a descent procedure to minimize~\eqref{eq:model-first-order} starting from \(s = 0\), any iterate, denoted generically as \(\hscp\), generated by that procedure will satisfy~\eqref{eq:xi-inequality}, i.e.,
\begin{equation}%
  \label{eq:hxikcp-inequality}%
  (\varphicp + \psi)(0; x)
  - (\varphicp + \psi)(\hscp; x) \geq \tfrac{1}{2} \nu^{-1} \|\hscp\|^2.
\end{equation}

Thus, an exact minimizer in~\eqref{eq:proximal-problem} would produce a Cauchy step \(\skcp\) that satisfies~\eqref{eq:xikcp-inequality}.
For brevity, we write \(\xikcp \defeq \xicp(\skcp; x_k, \nu_k^{-1})\) and \(\hxikcp\) instead of \(\hxicp(\hskcp; x_k, \nu_k^{-1})\).
The above shows that \(\xikcp \geq \tfrac{1}{2} \nu_k^{-1} \|\skcp\|^2\) and \(\hxikcp \geq \tfrac{1}{2} \nu_k^{-1} \|\hskcp\|^2\) provided \(\hskcp\) results in simple decrease in~\eqref{eq:model-first-order} from \(s = 0\).


\Cref{prop:stationarity} indicates that one role of the first-order models~\eqref{eq:phi-first-order}--\eqref{eq:model-first-order}, and hence of \(\hskcp\) and \(\hxikcp\) is to determine approximate stationarity.
The role of the second-order models~\eqref{eq:phi-second-order}--\eqref{eq:model-second-order} is to allow us to compute a step that improves upon the (inexact) Cauchy step.
Minimizing the second-order model is a well-defined problem for all sufficiently large \(\sigma_k\).

\begin{proposition}[{\protect \citealp[Proposition~\(3.3\)]{diouane-habiboullah-orban-2024}}]%
  Let \(\varphi(\cdot; x)\) be defined as in~\eqref{eq:phi-second-order}, and let \(\psi(\cdot; x)\) be proper, lsc and prox-bounded with threshold \(\lambda_x > 0\) and such that \(\psi(0; x) = h(x)\).
  For any \(\sigma > \lambda_x^{-1} - \lambda_{\min}(B(x))\), the set \(\argmin{s} m(s; x, \sigma)\) is nonempty and compact, where \(\lambda_{\min}\) represents the smallest eigenvalue.
\end{proposition}

\section{Algorithm and Convergence Analysis}%
\label{sec:analysis}

Our algorithm is a modification of method R2N of \citet{diouane-habiboullah-orban-2024}.
At a general iteration \(k\), an approximate Cauchy step \(\hskcp\) is computed together with the corresponding value of \(\hxikcp\) by minimizing~\eqref{eq:model-first-order} inexactly.
If \(x_k\) is not approximately stationary, a step \(\hsk\) is computed by approximately minimizing~\eqref{eq:model-second-order} in a sense made precise in \Cref{ass:inexact-decrease}.
Because only \(\hf\), and not \(f\), is available, we compute the ratio of achieved versus predicted decrease
\begin{equation}%
  \label{eq:def-rhok}
  \widehat{\rho}_k \defeq \frac{\hf(x_k) + h(x_k) - (\hf(x_k + \hsk) + h(x_k + \hsk))}{\varphi(0; x_k) + \psi(0; x_k) - (\varphi(\hsk; x_k) + \psi(\hsk; x_k))}
\end{equation}
to accept or reject \(\hsk\).
Acceptance of \(\hsk\) occurs when \(\widehat{\rho}_k \geq \widehat{\eta}_1 > 0\), which indicates that sufficient decrease occurs in \(\hf + h\).
The parameters of the algorithm, specifically \(\sigma_{\min}\), together with assumptions on the accuracy of \(\hf\), are chosen so that acceptance of \(\hsk\) also implies that sufficient decrease occurs in \(f + h\).
We then update \(\sigma_k\) accordingly, as in R2N.
All that is required of \(\hsk\) is that it satisfy a sufficient decrease condition---see \Cref{ass:cauchy-decrease} below.
That can be achieved, for instance, by computing \(\hskcp\) from a single (inexact) proximal-gradient iteration on~\eqref{eq:model-second-order} with a well-chosen step length \(\nu_k\) starting from \(s = 0\), and computing \(\hsk\) by continuing the (inexact) proximal-gradient iterations from \(\hskcp\).
Should \(\|\hsk\|\) be much larger than \(\|\hskcp\|\), we reset \(\hsk\) to \(\hskcp\) as in R2N.
The procedure is formally stated as \Cref{alg:ir2n}.
We refer the reader to \citep{diouane-habiboullah-orban-2024} for more background.

\begin{algorithm}[ht]%
  \caption{iR2N%
    \label{alg:ir2n}
  }
  \begin{algorithmic}[1]%
    \State Given \(\kf > 0\), \(\knf > 0\), choose constants \(0 < \gamma_3 \leq 1 < \gamma_1 \leq \gamma_2\), \(0 < \widehat{\eta}_1 \leq \widehat{\eta}_2 < 1\).
    \State Choose \(0 < \theta_1 < 1 < \theta_2\).
    \State Choose \(\sigma_{\min} > 4 \kappa_f / \widehat{\eta}_1\) and \(\sigma_0 \geq \sigma_{\min}\).
    \For{\(k = 0, 1, \ldots\)}
    \State Choose \(B_k \defeq B(x_k) \in \R^{n \times n}\) such that \(B_k = B_k^T\).
    \State%
    Set \(\nu_k \defeq \theta_1 / (\| B_k \| + \sigma_k)\).
    \Repeat
    \State Choose approximations \(\hf(x_k)\) and \(\hnf(x_k)\).
    \State%
    \label{stp:skcp}%
    Compute \(\hskcp\) an approximate solution of \(\min_s m_{\text{cp}}(s; x_k, \nu_k^{-1})\) and \(\hxikcp\).
    \State%
    \label{stp:sk}%
    Compute a step \(\hsk\) such that \(m(\hsk; x_k, \sigma_k) \leq m(\hskcp; x_k, \sigma_k)\).
    \If{%
      \label{stp:theta2}%
      \(\| \hsk \| > \theta_2 \| \hskcp \|\)}
    \State Reset \(\hsk = \hskcp\).
    \EndIf
    \Until{\(\hf(x_k)\) and \(\hnf(x_k)\) satisfy \Cref{ass:bound-inexact-f-nf}}.
    \State Compute the ratio \(\widehat{\rho}_k\) as in~\eqref{eq:def-rhok}.
    \If{\(\widehat{\rho}_k \geq \widehat{\eta}_1\)}
    \State Set \(x_{k+1} = x_k + \hsk\).
    \Else
    \State Set \(x_{k+1} = x_k\).
    \EndIf
    \State Update the regularization parameter according to
    \[
      \sigma_{k+1} \in
      \begin{cases}
        \begin{aligned}
           & [\gamma_3 \sigma_k, \, \sigma_k]          &  & \text{ if } \widehat{\rho}_k \geq \widehat{\eta}_2,                    &  & \qquad \text{\textcolor{gray}{very successful iteration}} \\
           & [\sigma_k, \, \gamma_1 \sigma_k]          &  & \text{ if } \widehat{\eta}_1 \leq \widehat{\rho}_k < \widehat{\eta}_2, &  & \qquad \text{\textcolor{gray}{successful iteration}}      \\
           & [\gamma_1 \sigma_k, \, \gamma_2 \sigma_k] &  & \text{ if } \widehat{\rho}_k < \widehat{\eta}_1                        &  & \qquad \text{\textcolor{gray}{unsuccessful iteration}}
        \end{aligned}
      \end{cases}
    \]
    \State Reset \(\sigma_{k+1} = \max(\sigma_{k+1}, \, \sigma_{\min})\).
    \EndFor
  \end{algorithmic}
\end{algorithm}

\subsection{Assumptions}

Intentionally, our assumptions are not the most general under which convergence of \Cref{alg:ir2n} can be shown to occur.
We have done so in order to highlight the influence of our assumptions on the inexactness of the objective, gradient and proximal operators evaluations on the analysis.
We refer the interested reader to \citep{diouane-habiboullah-orban-2024} for the current most general assumptions.
Nonetheless, we expect that our convergence guarantees remain valid under the weaker assumptions, at the cost of a more intricate analysis.

Our first assumption concerns Lipschitz-continuity of the gradient.
Technically, this assumption is only necessary for the complexity analysis; convergence can be guaranteed under continuous differentiability only.

\begin{assumption}
  \label{ass:lipschitz-gradient}
  \(\nabla f\) is Lipschitz-continuous with constant \(L \geq 0\)---see~\eqref{eq:lipschitz-gradient}.
\end{assumption}

We assume that \(\{B_k\}\) is bounded; a common assumption in the literature.
Under appropriate growth conditions, convergence is preserved even if \(\{B_k\}\) is allowed to grow unbounded \citep{diouane-habiboullah-orban-2024}.

\begin{assumption}
  \label{ass:hessian-bound}
  There exists \(\kappa_B > 0\) such that \(\|B_k\| \leq \kappa_B\) for all \(k\).
\end{assumption}

\Cref{ass:hessian-bound} is trivially satisfied when \(B_k = 0\), as in \citep[Algorithm~\(6.1\)]{aravkin-baraldi-orban-2022}.
It is also satisfied in \citep{becker-fadili-ochs-2019} where the objective is strongly convex and the model Hessian is defined by a positive definite limited-memory quasi-Newton update.
Under standard assumptions, the LBFGS and LSR1 updates satisfy \Cref{ass:hessian-bound} \citep{burdakov-gong-zikrin-yuan-2017,aravkin-baraldi-orban-2021}.

Our next assumption bounds the discrepancy between \(h\) and its model \(\psi\).

\begin{assumption}
  \label{ass:psi-accuracy}
  There exists \(\kappa_h > 0\) such that
  \(
  |\psi(x, s) - h(x + s)| \leq \kappa_h\|s\|^2
  \)
  for all \(x\) and \(s \in \R^n\).
\end{assumption}

The bound \(\|s\|^2\) in \Cref{ass:psi-accuracy} can be relaxed to \(o(\|s\|)\) \citep{diouane-habiboullah-orban-2024}.
\Cref{ass:psi-accuracy} is satisfied when \(\psi(s; x) = h(x + s)\), and when \(h(x) = g(c(x))\) where \(c\) is twice continuously differentiable with bounded second derivatives and \(g\) is globally Lipschitz continuous if we select \(\psi(s; x) = g(c(x) + \nabla c(x)^T s)\).

The next assumption drives the convergence analysis and states that the step \(\hsk\) computed at iteration \(k\) should result in a decrease at least comparable to that induced by the approximate Cauchy step in the first-order model.

\begin{assumption}
  \label{ass:cauchy-decrease}
  There is \(\theta_1 \in (0,1)\) such that
  \(
  \varphi(0;x_k) + \psi(0;x_k) - (\varphi(\hsk;x_k) + \psi(\hsk;x_k)) \geq (1-\theta_1)\hxikcp
  \)
  for all \(k\).
\end{assumption}

\noindent
As we now show, \Cref{ass:cauchy-decrease} holds for \(\hsk\) computed as stated in \Cref{alg:ir2n}.

\begin{lemma}%
  For \(\theta_1 \in (0,1)\) and \(\hsk\) as in \Cref{alg:ir2n}, \Cref{ass:cauchy-decrease} holds.
\end{lemma}

\begin{proof}
  The proof of \citep[Proposition~\(3.7\)]{diouane-habiboullah-orban-2024} applies with \(s = \hsk\) and \(\hskcp\) in place of \(\scp\).
  Indeed, it remains valid for any \(s \in \R^n\) and \(\scp \in \R^n\) as long as \(m(s; x, \sigma) \leq m(s_{\text{cp}}; x, \sigma)\), which is guaranteed by step 7 of \Cref{alg:ir2n}.
  \qed
\end{proof}

We ensure that Step 7 in \Cref{alg:ir2n} holds because the inexact Cauchy step \(\hskcp\) coincides with the first (inexact) step of the proximal gradient method applied to \(m(s; x_k, \sigma_k)\) from \(s=0\) with an appropriate step length \(\nu_k\).
Therefore, computing \(\hsk\) by continuing the proximal iterations from \(\hskcp\) leads to further decrease in \(m(s; x_k, \sigma_k)\).

The next assumption requires the norm of the computed step \(\hskcp\) to be at least a fraction of that of an exact step \(\skcp\).

\begin{assumption}
  \label{ass:inexact-decrease}
  There exists \(\kappa_s \in (0,1]\) such that, for all \(k\),
  \[
    \|\hskcp\| \geq \kappa_s \min \{ \|\skcp\| \mid \skcp \in \prox{\nu_k \psi(\cdot; x_k)} (-\nu_k \hnf(x_k)) \}.
  \]
\end{assumption}

In the experiments of \Cref{sec:implementation}, \(\psi(\cdot; x_k)\) satisfies the assumptions of \Cref{prop:skcp} and, therefore, the minimum in \Cref{ass:inexact-decrease} is well defined.

\Cref{ass:inexact-decrease} holds when \(\skcp\) is computed exactly, i.e., \(\hskcp = \skcp\).
Indeed, let \(\|s_{k,\min}\|\) be the smallest norm across all possible choices of \(\skcp\).
Several cases can occur.
Firstly, if \(\|s_{k,\min}\| > 0\), then \(\|\skcp\| > 0\) necessarily, and \Cref{ass:inexact-decrease} holds with \(\kappa_s \defeq \min(1, \, \|\skcp\| / \|s_{k,\min}\|)\).
If, on the other hand, \(\|s_{k,\min}\| = 0\), the same holds if we compute \(\skcp \neq 0\) but, should we compute \(\skcp = 0\), \Cref{prop:stationarity} would imply that \(x_k\) is stationary and the iterations would stop.
This case will be clarified in \Cref{lem:finite-succ}.

Though \Cref{ass:inexact-decrease} may appear as a simple decrease assumption, sufficient decrease is provided by~\eqref{eq:hxikcp-inequality}.

Details on how we satisfy \Cref{ass:inexact-decrease} when \(\hskcp \neq \skcp\) in certain situations relevant in practice can be found in \Cref{sec:implementation}.
We further comment on \Cref{ass:inexact-decrease} in \Cref{sec:discussion}.

In the same fashion as \citep{monnet-orban-2025}, we bound evaluation errors in terms of the step.
Similar assumptions are made in \citep{conn-gould-toint-2000} in a trust-region context.


\begin{assumption}%
  \label{ass:bound-inexact-f-nf}
  There exist \(\kappa_f > 0\) and \(\knf > 0\) such that, for all \(k \in \N\),
  \begin{align}%
    |f(x_k) - \hf(x_k)|               & \leq \kf \|\hsk\|^2, \\
    |f(x_k + \hsk) - \hf(x_k + \hsk)| & \leq \kf \|\hsk\|^2, \\
    \|\nabla f(x_k) - \hnf(x_k)\|     & \leq \knf \|\hsk\|.
  \end{align}
\end{assumption}

Finally, we assume that the objective is bounded below, which is only required in the complexity analysis.

\begin{assumption}
  \label{ass:fh-lower-bound}
  There exists \((f + h)_{\mathrm{low}} \in \R\) such that
  \(
  (f + h)(x) \geq (f + h)_{\mathrm{low}}
  \)
  for all \(x \in \R^n\).
\end{assumption}

\subsection{Convergence Analysis}

Our first result relates the decrease predicted by the model to the step size.

\begin{lemma}%
  \label{lem:xicp-snorm}
  Let \Cref{ass:cauchy-decrease} hold.
  Then,
  \[
    \varphi(0; x_k) + \psi(0; x_k) - (\varphi(s_k; x_k) + \psi(s_k; x_k)) \geq
    \tfrac{1}{2} \sigma_k \|\hsk\|^2.
  \]
\end{lemma}

\begin{proof}
  Line~\ref{stp:sk} of \Cref{alg:ir2n} ensures that \(m(\hsk; x_k, \sigma_k) \leq m(\hskcp; x_k, \sigma_k)\).
  By~\eqref{eq:hxikcp-inequality}, the value of \(\nu_k\) ensures that
  \[
    \hxikcp \geq \tfrac{1}{2} \nu_k^{-1} \|\hskcp\|^2 \geq \tfrac{1}{2} (\|B_k\| + \sigma_k) \|\hskcp\|^2 \geq \tfrac{1}{2} \hskcp^T (B_k + \sigma_k I) \hskcp,
  \]
  which is equivalent to saying that \(m(\hskcp; x_k, \sigma_k) \leq m(0; x_k, \sigma_k)\).
  Hence, \(\varphi(\hsk; x_k) + \psi(\hsk; x_k) + \tfrac{1}{2} \sigma_k \|\hsk\|^2 \leq \varphi(0; x_k) + \psi(0; x_k)\), which establishes the result.
  \qed
\end{proof}

Our next result mirrors \citep[Theorem 4.1]{aravkin-baraldi-orban-2024} and shows that whenever \(\sigma_k\) exceeds a threshold \(\sigma_\mathrm{succ}\), iteration \(k\) is very successful and \(\sigma_{k+1}\) decreases.

\begin{lemma}
  \label{lem:sigma-succ}
  Let \Cref{ass:lipschitz-gradient,ass:hessian-bound,ass:psi-accuracy,ass:cauchy-decrease,ass:bound-inexact-f-nf} be satisfied and define
  \[
    \sigma_{\mathrm{succ}} \defeq \max\left(
    \frac{ L + \kappa_B + 2 \kappa_h + 4 \kf + 2 \knf }{ 1 - \widehat{\eta}_2 },
    \, \lambda^{-1}
    \right) > 0.
  \]
  If, at iteration \(k\) of \Cref{alg:ir2n}, \(\hsk \neq 0\) and \(\sigma_k \geq \sigma_{\mathrm{succ}}\), then \(\widehat{\rho}_k \geq \widehat{\eta}_2\), and iteration \(k\) is very successful.
\end{lemma}

\begin{proof}
  As in the proof of \citep[Theorem~\(4.1\)]{aravkin-baraldi-orban-2024}, \(\sigma_k\) increases as long as it is below \(\lambda_{x_k}^{-1}\).
  Thus, we assume that \(\sigma_k \geq \lambda^{-1}\).
  The definitions of \(\widehat{\rho}_k\) and \(\varphi\), \Cref{ass:cauchy-decrease}, the triangle inequality and \Cref{lem:xicp-snorm} yield
  \begin{align*}
     & |\widehat{\rho}_k - 1|
    \notag                                                                                                                \\
     & = \frac{
           |\hf(x_k + \hsk) - \hf(x_k) - \hnf(x_k)^T \hsk - \tfrac{1}{2}\hsk^T B_k \hsk + h(x_k + \hsk) - \psi(\hsk; x_k)|
         }{
           \varphi(0;x) + \psi(0;x) - (\varphi(\hsk;x_k) + \psi(\hsk;x_k))
         }
    \notag
    \\
     & \leq
    \frac{
      |\hf(x_k + \hsk) - \hf(x_k) - \hnf(x_k)^T \hsk|
      +
      |\tfrac{1}{2} \hsk^T B_k \hsk|
      +
      |h(x_k + \hsk) - \psi(\hsk; x_k)|
    }{
      \tfrac{1}{2} \sigma_k \|\hsk\|^2
    }.
  \end{align*}

  The triangle inequality along with \Cref{ass:lipschitz-gradient,ass:bound-inexact-f-nf} bound the first term in the numerator as
  \begin{multline*}
    |\hf(x_k + \hsk) - \hf(x_k) - \hnf(x_k)^T \hsk| \\
    \leq |f(x_k + \hsk) - f(x_k) - \nf(x_k)^T \hsk| + 2 \kf \|\hsk\|^2 + \knf \|\hsk\|^2 \\
    \leq (\tfrac{1}{2} L + 2 \kf + \knf) \|\hsk\|^2.
  \end{multline*}
  \Cref{ass:hessian-bound} bounds the second term in the numerator by \(\tfrac{1}{2} \|B_k\| \|\hsk\|^2 \leq \tfrac{1}{2} \kappa_B \|\hsk\|^2\).
  \Cref{ass:psi-accuracy} bounds the last term in the numerator by \(\kappa_h \|\hsk\|^2\).
  After simplifying by \(\|\hsk\|^2\), those observations give
  \[
    |\widehat{\rho}_k - 1|
    \leq \frac{ L + \kappa_B + 2 \kappa_h + 4 \kf + 2 \knf }{ \sigma_k }.
  \]
  Therefore, \(\sigma_k \geq \sigma_{\text{succ}}\) implies that \(\widehat{\rho}_k \geq \widehat{\eta}_2\).
  \qed
\end{proof}

In \cref{lem:sigma-succ}, we showed that \(\sigma_k \geq \sigma_{\mathrm{succ}} \implies \widehat{\rho}_k \geq \widehat{\eta}_2\), which means that there is a decrease in \(\hf + h\).
Next, we show that a decrease occurs in \(f + h\) every time a step is accepted.

\begin{lemma}%
  \label{lem:link-eta_2-inexact}
  Let \Cref{ass:bound-inexact-f-nf,ass:cauchy-decrease} hold.
  At iteration \(k\), denote
  \[
    \rho_k \defeq \frac{f(x_k) + h(x_k) - (f(x_k + \hsk) + h(x_k + \hsk))}{\varphi(0; x_k) + \psi(0; x_k) - (\varphi(\hsk; x_k) + \psi(\hsk; x_k))}
  \]
  the measure of agreement between the actual and predicted decrease in \(f + h\).
  Let \(\sigma_{\min}\) be as in \Cref{alg:ir2n} and
  \[
    \eta_1 \defeq \widehat{\eta}_1 - \frac{ 4 \kf }{ \sigma_{\min} } > 0,
    \qquad
    \eta_2 \defeq \widehat{\eta}_2 - \frac{ 4 \kf }{ \sigma_{\min} } > 0.
  \]
  Then, \(\widehat{\rho}_k \geq \widehat{\eta}_1\) \(\Longrightarrow\) \(\rho_k \geq \eta_1\) and \(\widehat{\rho}_k \geq \widehat{\eta}_2\) \(\Longrightarrow\) \(\rho_k \geq \eta_2\).
\end{lemma}

\begin{proof}
  By definition of \(\widehat{\rho}_k\) and \(\rho_k\),
  \begin{equation*}
    \widehat{\rho}_k = \rho_k + \frac{(\hf - f)(x_k) + (f - \hf)(x_k + \hsk)}{(\varphi + \psi)(0; x_k) - (\varphi + \psi)(\hsk; x_k)}.
  \end{equation*}
  Because \Cref{alg:ir2n} enforces \(\sigma_k \geq \sigma_{\min} > 0\), \Cref{lem:xicp-snorm} and \Cref{ass:bound-inexact-f-nf} give
  \begin{equation*}
    \lvert \widehat{\rho}_k - \rho_k \rvert \leq
    \frac{
      2 \kappa_f \|\hsk\|^2
    }{
      \tfrac{1}{2} \sigma_k \|\hsk\|^2
    }
    \leq
    \frac{
      4 \kf
    }{
      \sigma_{\min}
    }.
  \end{equation*}
  Now, if \(\widehat{\rho}_k \geq \widehat{\eta}_1\),
  \[
    \rho_k \geq \widehat{\eta}_1 - \frac{ 4 \kf }{ \sigma_{\min} } = \eta_1.
  \]
  The lower bound on \(\sigma_{\min}\) ensures \(\eta_1 > 0\).
  The same holds for \(\eta_2\) because \(\widehat{\eta}_2 \geq \widehat{\eta}_1\).
  \qed
\end{proof}

\noindent
\Cref{lem:sigma-succ,lem:link-eta_2-inexact} together imply that \(\widehat{\rho}_k \geq \widehat{\eta}_1\) guarantees a decrease in \(f + h\).

The next result is classic and considers the case where only a finite number of successful iterations occur.

\begin{lemma}
  \label{lem:finite-succ}
  Let \Cref{ass:lipschitz-gradient,ass:hessian-bound,ass:psi-accuracy,ass:cauchy-decrease,ass:bound-inexact-f-nf} be satisfied.
  Suppose \Cref{alg:ir2n} generates finitely many successful iterations.
  Then \(x_k = x_\star\) and \(\hskcp = 0\) for all \(k\) sufficiently large.
  If, in addition, \Cref{ass:inexact-decrease} holds, \(x_\star\) is stationary.
\end{lemma}

\begin{proof}
  By assumption, there is \(k_0 \in \N\) such that \(x_k = x_\star\) for all \(k \geq k_0\).
  If \(\hskcp \neq 0\), as of iteration \(k_0\), \Cref{alg:ir2n} repeatedly computes nonzero steps \(s_k\), all of which are rejected, i.e., \(\widehat{\rho}_k < \widehat{\eta}_1\).
  Thus, for all \(k \geq k_0\), \(\sigma_{k+1} \geq \gamma_1 \sigma_k\).
  Hence, for sufficiently large \(k\), \(\sigma_k > \sigma_{\text{succ}}\), which triggers a successful iteration, and is absurd.

  Under \Cref{ass:inexact-decrease}, \(\hskcp = 0\) implies that \(0 \in \prox{\nu_k \psi(\cdot; x_k)}(-\nu_k \hnf(x_k))\).
  Line~\ref{stp:theta2} of \Cref{alg:ir2n} then implies that \(s_k = 0\), and \Cref{ass:bound-inexact-f-nf} implies that \(\hnf(x_k) = \nabla f(x_k)\).
  Thus, \(x_\star\) is stationary.
  \qed
\end{proof}

\Cref{lem:sigma-succ} implies that there exists \(\sigma_{\max} = \min(\sigma_0, \gamma_2 \sigma_{\mathrm{succ}})\) such that \(\sigma_k \leq \sigma_{\max}\) for all \(k \in \N\).
Consequently, \Cref{ass:hessian-bound} yields that for all \(k \in \N\),
\begin{equation}%
  \label{eq:nuk-bounds}
  \nu_{\min} \leq \nu_k \leq \nu_{\max},
  \quad
  \nu_{\min} \defeq \theta_1 / (\kappa_B + \sigma_{\max}),
  \quad
  \nu_{\max} \defeq \theta_1 / \sigma_{\min}.
\end{equation}

Let \(\epsilon > 0\).
We seek a bound on \(k_{\epsilon} \defeq \min\{ k \in \mathbb{N} \mid \nu_k^{-1} \|\hskcp\| < \epsilon \} = |\mathcal{S}(\epsilon)| + |\mathcal{U}(\epsilon)| + 1\), where
\[
  \mathcal{S}(\epsilon) \defeq \{ k \in \N \mid \widehat{\rho}_k \geq \widehat{\eta}_1 \text{ and } k < k_{\epsilon} \}, \mathhfill
  \mathcal{U}(\epsilon) \defeq \{ k \in \N \mid \widehat{\rho}_k < \widehat{\eta}_1 \text{ and } k < k_{\epsilon} \}.
\]

\begin{lemma}
  \label{lem:inf-succ}
  Let \Cref{ass:lipschitz-gradient,ass:hessian-bound,ass:psi-accuracy,ass:cauchy-decrease,ass:fh-lower-bound,ass:bound-inexact-f-nf} be satisfied.
  Assume that \Cref{alg:ir2n} generates infinitely many successful iterations.
  Then,
  \[
    |\mathcal{S}(\epsilon)| \leq \frac{(f+h)(x_0) - (f + h)_{\mathrm{low}}}{\tfrac{1}{2}\eta_1 (1-\theta_1) \nu_{\min}} \ \epsilon^{-2} \defeq \omega_s \epsilon^{-2},
  \]
  where \(\nu_{\min}\) is defined in~\eqref{eq:nuk-bounds}.
\end{lemma}

\begin{proof}
  Let \(k \in \mathcal{S}(\epsilon)\).
  By definition, \(\widehat{\rho}_k \geq \widehat{\eta}_1\), which, by \Cref{lem:link-eta_2-inexact}, implies that \(\rho_k \geq \eta_1\).
  \Cref{ass:cauchy-decrease},~\eqref{eq:nuk-bounds},~\eqref{eq:hxikcp-inequality} and the fact that \(k < k_\epsilon\) then imply
  \begin{align*}
    (f + h)(x_k) - (f + h)(x_k + s_k) & \geq \eta_1 ((\varphi + \psi)(0; x_k) - (\varphi + \psi)(s_k; x_k))
    \\
                                      & \geq \eta_1 (1-\theta_1) \hxikcp
    \\
                                      & \geq \tfrac{1}{2}\eta_1 (1-\theta_1) \nu_k^{-1} \|\hskcp\|^2
    \\
                                      & \geq \tfrac{1}{2}\eta_1 (1-\theta_1) \nu_k \epsilon^2
    \\
                                      & \geq \tfrac{1}{2}\eta_1 (1-\theta_1) \nu_{\min} \epsilon^2.
  \end{align*}
  The rest of the proof is classic and identical to, e.g., \citep[Lemma~\(4.3\)]{aravkin-baraldi-orban-2024}.
  \qed
\end{proof}

It is remarkable that the bound in \Cref{lem:inf-succ} is identical to that of the standard R2N, which is more apparent when comparing with \citep[Lemma~\(4.3\)]{aravkin-baraldi-orban-2024} than with \citep[Theorem~\(6.4\)]{diouane-habiboullah-orban-2024}.
The extra factor \(\tfrac{1}{2}\) in the denominator of our bound on \(|\mathcal{S}(\epsilon)|\) is due to the fact that we use \(\nu_k^{-1} \|\hskcp\|\) as stationarity measure instead of \(\nu_k^{-1/2} \hxikcp^{1/2}\), as in \citep{aravkin-baraldi-orban-2024}.

Finally, we recover a worst-case complexity bound of the same order as in the analysis with exact proximal operator evaluations.
The proof is identical to that of, e.g., \citep[Theorem~\(4.5\)]{aravkin-baraldi-orban-2024}, and is omitted.

\begin{theorem}%
  \label{theorem:total-iterations}
  Let \Cref{ass:lipschitz-gradient,ass:hessian-bound,ass:psi-accuracy,ass:cauchy-decrease,ass:fh-lower-bound,ass:bound-inexact-f-nf} be satisfied.
  Then,
  \[
    |\mathcal{S}(\epsilon)| + |\mathcal{U}(\epsilon)| = \left(1 + |\log_{\gamma_1}(\gamma_3)|\right) \omega_s \epsilon^{-2}  + \log_{\gamma_1}(\sigma_{\max}/\sigma_0) = O(\epsilon^{-2}),
  \]
  where \(\omega_s\) is defined in \Cref{lem:inf-succ}.
\end{theorem}

\Cref{theorem:total-iterations} shows that iR2N brings the measure \(\nu_k^{-1} \|\hskcp\|\) below \(\epsilon\) in \(O(\epsilon^{-2})\) iterations.
That measure is not a stationarity measure because it includes the inexactness on \(\hskcp\).
By \Cref{ass:inexact-decrease}, there exists an exact Cauchy step \(\skecp\) such that
\begin{equation}
  \label{eq:link-inexact-measure}
  \nu_k^{-1} \|\skecp\| \leq \kappa_s^{-1} \nu_k^{-1} \|\hskecp\| < \kappa_s^{-1} \epsilon.
\end{equation}
Thus, if  \(\nu_k^{-1} \|\hskecp\|\) is small, \(\nu_k^{-1} \|\skecp\|\) is comparably small.
The next result shows that when the latter occurs, we have identified a near stationary point, and marks the impact of \(\kappa_s\) on the analysis.

\begin{theorem}%
  \label{theorem:eps-stationarity}
  Let \Cref{ass:bound-inexact-f-nf,ass:inexact-decrease} be satisfied.
  Let \(\epsilon > 0\) and assume \(\nu_k^{-1} \|\hskcp\| < \epsilon\).
  There exists \(\skcp \in \prox{\nu \psi(\cdot; x_k)} (-\nu_k \hnf(x_k))\) that satisfies \Cref{ass:inexact-decrease} such that \(\|\skcp\| < \kappa_s^{-1} \nu_{\max} \epsilon\), and \(u_k \in \nf(x_k) + \partial \psi(\skcp; x_k)\) such that
  \begin{equation}%
    \label{eq:uk-norm}
    \|u_k\| < \left( \kappa_{\nabla} \theta_2 \nu_{\max} + \kappa_s^{-1} \right) \,  \epsilon.
  \end{equation}
\end{theorem}

\begin{proof}
  By definition, \(\skcp\) is an exact minimizer of~\eqref{eq:model-first-order}, thus
  \begin{align}
    0 & \in \hnf(x_k) + \nu_k^{-1}\skcp + \partial \psi(\skcp; x_k) \nonumber \\
      & = \nf(x_k) + g_k + \nu_k^{-1}\skcp + \partial \psi(\skcp; x_k),
    \label{eq:skcp-subdifferential}
  \end{align}
  where \(g_k \defeq \hnf(x_k) - \nabla f(x_k)\) and \(\|g_k\| \leq \kappa_\nabla \|\hsk\| \leq \kappa_\nabla \theta_2 \|\hskcp\|\) from \Cref{ass:bound-inexact-f-nf} and line~\ref{stp:theta2} of \Cref{alg:ir2n}.
  By~\eqref{eq:nuk-bounds} and \(\nu_k^{-1} \|\hskcp\| < \epsilon\),
  \(
  \|\hskcp\| \leq \nu_k \epsilon < \nu_{\max} \epsilon.
  \)
  Thus, \(\|g_k\| < \kappa_{\nabla} \theta_2 \nu_{\max} \epsilon\).

  On the other hand, \Cref{ass:inexact-decrease} gives
  \begin{equation*}%
    \|\nu_k^{-1} \skcp\| \leq \kappa_s^{-1} \nu_k^{-1} \|\hskcp\| < \kappa_s^{-1} \epsilon.
  \end{equation*}
  Now,~\eqref{eq:skcp-subdifferential} implies that
  \[
    u_k \defeq -(g_k + \nu_k^{-1}\skcp) \in \nf(x_k) + \partial \psi(\skcp; x_k).
  \]
  Because \(\|u_k\| \leq \|g_k\| + \|\nu_k^{-1} \skcp\|\),~\eqref{eq:uk-norm} holds.
  Finally, the same reasoning as above shows that \(\|\skcp\|\) is bounded as announced.
  \qed
\end{proof}

The following results directly from \Cref{theorem:total-iterations} and mirrors \citep[Lemma~\(3\)]{leconte-orban-2023}.

\begin{lemma}%
  \label{lem:s-to-zero}
  Under the assumptions of \Cref{theorem:total-iterations} and \Cref{ass:inexact-decrease}, there exists an infinite index set \(N \subseteq \N\) and \(\{\skcp\}\) where \(\skcp \in \prox{\nu \psi(\cdot; x_k)} (-\nu_k \hnf(x_k))\) for all \(k\) such that
  \begin{enumerate}
    \item%
      \label{itm:scp-to-0}%
      \(\{\hskcp\}_N \to 0\) and \(\{\skcp\}_N \to 0\),
    \item%
      \label{itm:s-to-0}%
      \(\{\hsk\}_N \to 0\)
    \item%
      \label{itm:sub-opt}%
      there exists \(u_k \in \nabla f(x_k) + \partial \psi(\skcp; x_k)\) such that \(\{u_k\}_N \to 0\).
  \end{enumerate}
\end{lemma}

\begin{proof}
  Claim~\ref{itm:scp-to-0} follows directly from \Cref{theorem:total-iterations},~\eqref{eq:nuk-bounds} and~\eqref{eq:link-inexact-measure}.
  Claim~\ref{itm:s-to-0} follows from Line~\ref{stp:theta2} of \Cref{alg:ir2n}.
  Claim~\ref{itm:sub-opt} results from \Cref{theorem:eps-stationarity}.
  \qed
\end{proof}

We close this section with a result stating that every limit point of the sequence \(\{x_k\}_N\) generated by \Cref{alg:ir2n} is stationary, where \(N\) is defined in \Cref{lem:s-to-zero}, under an assumption on the subdifferential of the models \(\psi(\cdot; x_k)\).

Recall that for a sequence of sets \(\{\mathcal{A}_k\}\) with \(\mathcal{A}_k \subseteq \R^n\) for all \(k \in \N\), the set \(\limsup \mathcal{A}_k\) is the set of limits of all possible convergent sequences \(\{a_k\}_N\) with \(N \subset \N\) infinite and \(a_k \in \mathcal{A}_k\) for all \(k \in N\).

\begin{theorem}%
  Under the assumptions of \Cref{theorem:total-iterations}, \Cref{ass:inexact-decrease,ass:psi-prox-bounded}, let \(N \subseteq \N\) be as in \Cref{lem:s-to-zero}.
  Assume that \(\{x_k\}_N \to \bar{x}\) and that
  \begin{equation}%
    \label{eq:limsup-subdifferential}
    \limsup_{k \in N} \partial \psi(\skcp; x_k) \subseteq \partial \psi(0; \bar{x}).
  \end{equation}
  Then \(\bar{x}\) is stationary for~\eqref{eq:problem-adressed}.
\end{theorem}

\begin{proof}
  By our assumptions, \Cref{lem:s-to-zero}, continuity of \(\nabla f\) and \Cref{ass:psi-prox-bounded},
  \[
    0 \in \nabla f(\bar{x}) + \limsup_{k \in N} \partial \psi(\skcp; x_k) \subseteq \nabla f(\bar{x}) + \partial \psi(0; \bar{x}) \subseteq \nabla f(\bar{x}) + \partial h(\bar{x}).
  \]
  Thus, \(\bar{x}\) is stationary for~\eqref{eq:problem-adressed}.
  \qed
\end{proof}

As \citet{leconte-orban-2023} explain,~\eqref{eq:limsup-subdifferential} holds in several relevant cases, e.g.,
\begin{enumerate}
  \item each \(\psi(\cdot; x_k)\) and \(\psi(\cdot; \bar{x})\) are proper, lsc and convex with \(\psi(\cdot; x_k) \to \psi(\cdot; \bar{x})\) in the epigraphical sense, and \(0 \in \mathrm{dom} \ \psi(\cdot; \bar{x})\);
  \item \(\psi(s; x) \defeq h(x + s)\) and \(h(x_k + \skcp) \to h(\bar{x})\) as would occur, in particular but not exclusively, when \(h\) is continuous.
\end{enumerate}

\section{Evaluation of inexact proximal operators}%
\label{sec:implementation}

In this section, we discuss the practical implementation of \Cref{alg:ir2n} with focus on computing an approximate solution of~\eqref{eq:proximal-problem} that satisfies \Cref{ass:inexact-decrease}.
Our approach is simple: assume that an upper bound \(M_k > 0\) on \(\|\skcp\|\) can be determined based on properties of \(\psi(\cdot; x_k)\).
Assume also that a descent procedure is applied to~\eqref{eq:proximal-problem} starting from \(s = 0\) that generates iterates \(\widehat{s}_j\), \(j \geq 0\).
Then, stopping the procedure as soon as \(\|\widehat{s}_j\| \geq \kappa_s M_k\) ensures that \Cref{ass:inexact-decrease} holds.

As becomes apparent in the examples below, \(\kappa_s\) is chosen a priori, but its value may influence performance.
Indeed, \(\kappa_s \approx 0\) suggests that little effort will be expended in computing a step, at the risk of computing a low-quality step that may result in more iterations.
Conversely, \(\kappa_s \approx 1\) suggests that more effort will be invested in computing a better-quality step, at the risk of spending more time than necessary in the proximal operator evaluations.
Thus, a (possibly application-dependent) balance must be found in the choice of \(\kappa_s\).

We consider three regularizers whose proximal operators~\eqref{eq:def-prox} are not known analytically and must be computed inexactly:
\begin{align}%
  \label{eq:inexact_prox-lp}
  h(x) & = \ell_p(x)      = \|x\|_p                          & (1 \leq p < \infty), \\
  \label{eq:inexact_prox-tvp}
  h(x) & = \text{TV}_p(x) = (\sum_i |x_i - x_{i-1}|^p)^{1/p} & (1 \leq p < \infty), \\
  \label{eq:inexact_prox-chi}
  h(x) & = \chi_{p,r}(x)  =
  \begin{cases}
    0      & \text{if } \|x\|_p^p \leq r \\
    \infty & \text{otherwise}
  \end{cases}
       & (0 < p < 1),
\end{align}
where \(\text{TV}_p\) is the one-dimensional total-variation operator, and \(\chi_{p,r}\) is the indicator of the \(\ell_p\)-pseudo norm ``ball'' of radius \(r^{1/p}\) for \(r > 0\).

The next lemmas derive bounds on the norm of solutions to the proximal problems associated with those regularizers.

\begin{lemma}%
  \label{lem:bounds-on-lp}
  Let \(h\) be given by~\eqref{eq:inexact_prox-lp} and \(\psi(s; x_k) \defeq h(x_k + s)\) with \(s \in \R^n\).
  The unique solution \(\skcp\) of~\eqref{eq:proximal-problem} is such that
  \begin{equation}%
    \label{eq:bounds-on-lp}
    \|\skcp\| \leq
    \begin{cases}
      \nu_k (\|\hnf(x_k)\| + n^{1/p - 1/2}) & (1 \leq p < 2) \\
      \nu_k (\|\hnf(x_k)\| + 1)             & (p \geq 2).
    \end{cases}
  \end{equation}
\end{lemma}

\begin{proof}
  Since \(\psi(\cdot; x_k)\) is convex,~\eqref{eq:proximal-problem} is strongly convex and, therefore, has a unique solution \(\skcp\).
  The necessary optimality conditions read
  \[
    \hnf(x_k) + \nu_k^{-1}\skcp + u_k = 0, \qquad u_k \in \partial \psi(\skcp;x_k).
  \]
  Here,
  \(
  \partial \psi(\skcp; x_k) = \{ u \in \R^n \mid \|u\|_q \leq 1 \text{ and } u^T (\skcp + x_k) = \|\skcp + x\|_p \},
  \)
  where \(q\) is such that \(1/p + 1/q = 1\).
  %
  By equivalence of norms,
  \[
    \|u_k\| \leq n^{1/2 - 1/q} \, \|u_k\|_q \leq n^{1/2 - 1/q} = n^{1/p - 1/2}.
  \]
  When \(1 \leq p \leq 2\), the latter bound is attained for \(u_k \defeq (n^{-1/q}, n^{-1/q}, \dots, n^{-1/q})\) with \(\|u_k\|_q = 1\).
  When \(p > 2\), the bound simplifies to \(\|u_k\| \leq 1\), which is attained for \(u_k \defeq (1, 0, \dots, 0)\).
  Thus, \(\|\skcp\| = \nu_k \|\hnf(x_k) + u_k\| \leq \nu_k (\|\hnf(x_k)\| + \|u_k\|)\), which yields~\eqref{eq:bounds-on-lp}.
  \qed
\end{proof}

The next result helps bound solutions of~\eqref{eq:proximal-problem} when \(h\) is given by~\eqref{eq:inexact_prox-tvp}, but is more general, which is why it is stated separately.

\begin{lemma}%
  \label{lem:bound-on-matrix-regularizer}
  Let \(A \in \R^{m \times n}\), \(h(x) \defeq \|Ax\|_\bullet\) where \(\|\cdot\|_\bullet\) is a norm on \(\R^m\), and \(\psi(s; x_k) \defeq h(x_k + s)\).
  The unique solution \(\skcp\) of~\eqref{eq:proximal-problem} satisfies
  \begin{equation}%
    \label{eq:bound-on-matrix-regularizer}
    \|\skcp\| \leq \nu_k \left(\|\hnf(x_k)\| + \|A\| \, \|u_k\| \right),
  \end{equation}
  where \(u_k \in \partial \|A(x_k + \skcp)\|_\bullet = \partial h(x_k + \skcp)\).
\end{lemma}

\begin{proof}
  Here again, \(\skcp\) is unique by strong convexity of~\eqref{eq:proximal-problem}.
  For \(\eta(y) \defeq \|y\|_{\bullet}\),
  \[
    \partial \eta(y) = \{ u \in \R^m \mid \|u\|_\star \leq 1 \text{ and } u^T y = \|y\|_{\bullet} \},
  \]
  where \(\|\cdot\|_\star\) is the dual norm of \(\|\cdot\|_{\bullet}\).
  By \citep[Theorem~\(23,9\)]{rockafellar-1970}, \(\partial \psi(s; x_k) = A^T \partial \eta(A (x_k + s))\).
  Thus, the first-order optimality conditions of~\eqref{eq:proximal-problem} imply
  \[
    0 \in \hnf(x_k) + \nu_k^{-1} \skcp + A^T u_k,
  \]
  where \(u_k \in \partial \eta(A (x_k + \skcp))\).
  We extract \(\skcp = -\nu_k (\hnf(x_k) + A^T u_k)\), and \(\|\skcp\| \leq \nu_k (\|\hnf(x_k)\| + \|A^T\| \, \|u_k\|)\), which is~\eqref{eq:bound-on-matrix-regularizer} since \(\|A\| = \|A^T\|\).
  \qed
\end{proof}

\Cref{lem:bound-on-matrix-regularizer} does not state a bound on \(\|u_k\|\) as one would depend on \(\|\cdot\|_{\bullet}\) and the bound \(\|u_k\|_\star \leq 1\).
The next corollary applies \Cref{lem:bound-on-matrix-regularizer} to~\eqref{eq:inexact_prox-tvp}.

\begin{corollary}%
  Let \(h\) be as in~\eqref{eq:inexact_prox-tvp} and \(\psi(s; x_k) \defeq h(x_k + s)\).
  The unique solution \(\skcp\) of~\eqref{eq:proximal-problem} satisfies
  \begin{equation}
    \label{eq:bound-on-tvp}
    \|\skcp\| \leq
    \begin{cases}
      \nu_k \left(\|\hnf(x_k)\| + 2\sin \left(\tfrac{\pi(n-1)}{2n}\right) n^{1/p - 1/2}\right) & (1 \leq p < 2) \\
      \nu_k \left(\|\hnf(x_k)\| + 2\sin \left(\tfrac{\pi(n-1)}{2n}\right)\right)               & (p \geq 2).
    \end{cases}
  \end{equation}
\end{corollary}

\begin{proof}
  Apply \Cref{lem:bound-on-matrix-regularizer} with \(\|\cdot\|_{\bullet} = \|\cdot\|_p\) and
  \[
    A \defeq
    \begin{bmatrix}
      -1 & 1      &        &   \\
         & \ddots & \ddots &   \\
         &        & -1     & 1
    \end{bmatrix}
    \in \R^{(n-1) \times n}.
  \]
  Note that \(A^T A\) is the centered finite-difference operator for second derivatives, which is symmetric, tridiagonal and Toeplitz.
  Its eigenvalues are thus known in closed form, hence the value of \(\|A\|\) \citep[p.~\(54\)]{smith-1985}.
  Finally, \(\|u_k\|\) can be bounded as in the proof of \Cref{lem:bounds-on-lp}.
\end{proof}

The final lemma derives a bound on the solution of the proximal problem associated to the indicator function in~\eqref{eq:inexact_prox-chi}.

\begin{lemma}%
  \label{lem:bound-on-chi}
  Let \(h\) be as in~\eqref{eq:inexact_prox-chi} and \(\psi(s; x_k) \defeq h(x_k + s)\).
  Any solution \(\skcp\) of~\eqref{eq:proximal-problem} satisfies
  \begin{equation}%
    \label{eq:bound-on-chi}
    \|\skcp\| \leq r^{1/p} + \|x_k\|.
  \end{equation}
\end{lemma}

\begin{proof}
  Because \(0 < p < 1\), \(t \mapsto t^p\) is concave for \(t \geq 0\), and thus subadditive, i.e., \((a + b)^p \leq a^p + b^p\) for any \(a\), \(b \geq 0\).
  Let \(u \in \R^n\).
  By recurrence on \(n\),
  \(
  \|u\|_p^p = \sum_{i=1}^n |u_i|^p \geq ( \sum_{i=1}^n |u_i| )^p,
  \)
  which states that \(\|u\|_p \geq \|u\|_1\).
  This implies that the unit ``ball'' in \(\ell_p\)-pseudo-norm is a subset of the unit \(\ell_1\)-norm ball.
  In turn, the latter is a subset of the unit \(\ell_2\)-norm ball.
  A scaling argument shows that the same holds with balls of radius \(r > 0\).
  Therefore, because \(\|x_k + \skcp\|_p \leq r^{1/p}\), we have \(\|x_k + \skcp\| \leq r^{1/p}\).
  The triangle inequality yields
  \(
  \|\skcp\| \leq \|x_k + \skcp\| + \|x_k\| \leq r^{1/p} + \|x_k\|.
  \)
  \qed
\end{proof}

In~\eqref{eq:bounds-on-lp},~\eqref{eq:bound-on-tvp}, and~\eqref{eq:bound-on-chi}, the bound on \(\|\skcp\|\) depends only on known quantities at iteration \(k\).
Thus, we can enforce \Cref{ass:inexact-decrease} by stopping the inexact proximal procedure as soon as \(\|\hskcpj\|\) exceeds a fixed fraction of said bound.

\section{Further examples}

The examples of \Cref{sec:implementation} are representative but not exhaustive.
To emphasize the breadth of our approach, we now turn to regularizers that appear in recent exact-penalty \citep{diouane-gollier-orban-2024} and trust-region \citep{aravkin-baraldi-orban-2022,aravkin-baraldi-orban-2024} methods.
In each case, the results of \Cref{sec:implementation}, or a minor variation thereof, provide the bound needed to enforce \Cref{ass:inexact-decrease}.

\subsection{\(\ell_2\)-norm penalty of a linear form}

\citet{diouane-gollier-orban-2024} devise an exact penalty method for equality-constrained smooth optimization in which \(\psi(s; x_k) = \tau_k \|A_k s + b_k\|\), where \(\tau_k > 0\) is a penalty parameter, \(\|\cdot\|\) is a norm, \(A_k\) is the \(m \times n\) constraint Jacobian about \(x_k\), \(b_k \in \R^m\), and \(m \leq n\).
For any matrix \(M\), \(M^\dagger\) is its Moore-Penrose pseudo-inverse.

When \(\|\cdot\| = \|\cdot\|_2\), the proximal operator of \(\nu_k \psi(\cdot; x_k)\) is known, but requires an iterative procedure.

\begin{proposition}[{\protect \citealp[Corollary~\(5.3\)]{diouane-gollier-orban-2024}}]%
  \label{prop:more-sorensen}
  Let $A_k \in \R^{m \times n}$, $b_k \in \R^m$, $\tau_k > 0$, and \(\psi(s; x_k) \coloneq \tau_k \|A_k s + b_k\|_2\).
  For \(\nu_k > 0\), the only element in \(\prox{\nu_k \psi(\cdot; x_k)}(w)\) is
  \begin{equation*}%
    \begin{cases}%
      w - A_k^T {(A_k A_k^T)}^\dagger (A_k w + b_k) \quad           & \textup{if} \ \| {(A_k A_k^T)}^\dagger (A_k w + b_k)\|_2 \leq \nu_k \tau_k \\
                                                                    & \textup{and} \ A_k w + b_k \in \textup{Range}(A_k A_k^T)                   \\
      w - A_k^T {(A_k A_k^T + \alpha^* I)}^{-1} (A_k w + b_k) \quad & \textup{otherwise},
    \end{cases}
  \end{equation*}
  where $\alpha^*$ is the unique positive root of the strictly decreasing function
  \begin{equation*}%
    g(\alpha) = \|{(A_k A_k^T + \alpha I)}^{-1} (A_k w + b_k)\|_2^2 - \nu_k^2 \tau_k^2.
  \end{equation*}
\end{proposition}

Whether \(A_k\) has full row rank or not, \Cref{prop:more-sorensen} yields the following.

\begin{lemma}%
  \label{lem:more-sorensen}
  Let $A_k \in \R^{m \times n}$, $b_k \in \R^m$, $\tau_k > 0$, and \(\psi(s; x_k) \coloneq \tau_k \|A_k s + b_k\|_2\).
  The only solution to~\eqref{eq:def-prox} satisfies
  \[
    \|\skcp\| \leq \nu_k \|\hnf(x_k)\| + \|A_k\|_2 \, \nu_k \tau_k.
  \]
\end{lemma}

\begin{proof}
  We apply \Cref{prop:more-sorensen} with \(w \defeq -\nu_k \hnf(x_k)\).
  The first case occurs if \(\|(A_k A_k^T)^\dagger (A_k w + b_k)\|_2 \leq \nu_k \tau_k\).
  In the second case, \(g(\alpha^*) = 0\), which means that \(\|(A_k A_k^T + \alpha^* I)^{-1} (A_k w + b_k)\|_2 = \nu_k \tau_k\).
  Thus, in both cases, the triangle inequality shows that \(\skcp\) is bounded as announced.
  \qed
\end{proof}

\citet{diouane-gollier-orban-2024} explain that \(\skcp\) can be found by first checking whether the first case occurs, and, if it does not, by applying Newton's method with safeguards to find the root of \(g\).
The pseudo-inverse and inverse are only computed implicitly by solving systems of the form
\[
  \begin{bmatrix}
    I   & \phantom{-}A_k^T \\
    A_k & -\alpha I
  \end{bmatrix}
  \begin{bmatrix}
    s \\
    z
  \end{bmatrix}
  = -
  \begin{bmatrix}
    w \\
    b_k
  \end{bmatrix}
  ,
\]
for appropriate values of \(\alpha > 0\).
The Newton iterations can be stopped as soon as they generate an iterate whose norm is a fraction \(\kappa_s\) of the bound given in \Cref{lem:more-sorensen}, or an accurate solution has been found.

\subsection{Exact penalty in arbitrary norm}

\Cref{prop:more-sorensen} is a special case of \citep[Theorem~\(5.1\)]{diouane-gollier-orban-2024}, which states the following.
Let \(\psi(s; x_k) = \tau_k \|A_k s + b_k\|_\bullet\), where \(\|\cdot\|_\bullet\) is an arbitrary norm.
Then, \(\skcp = \nu_k (A_k^T y - \hnf(x_k))\), where \(y\) minimizes a certain quadratic subject to the trust-region constraint \(\|y\|_\star \leq \tau_k\), and \(\|\cdot\|_\star\) is the dual norm.
The triangle inequality yields \(\|\skcp\| \leq \nu_k (\|A_k^T y\| + \|\hnf(x_k)\|)\).
A bound on \(\|A_k^T y\|\) will depend on the norm chosen.

A common choice is \(\|\cdot\|_\bullet = \|\cdot\|_1\), in which case \(\|\cdot\|_\star = \|\cdot\|_\infty\).
Thus, we have the bound \(\|A_k^T y\| \leq \|A_k\| \, \|y\| \leq \sqrt{n} \tau_k \|A_k\|\).
A tighter bound that does not depend on \(n\) may be stated as \(\|A_k^T y\| \leq \tau_k \|A_k\|_{2,\infty}\), where
\[
  \|A_k\|_{2, \infty} \defeq \max \{ \|A_k y\|_2 \mid \|y\|_\infty \leq 1 \}.
\]

Similar bounds follow from the other common choice \(\|\cdot\|_\bullet = \|\cdot\|_\infty\), for which \(\|\cdot\|_\star = \|\cdot\|_1\).

\subsection{Trust-region indicator}

In the context of trust-region methods for~\eqref{eq:problem-adressed}, \citet{aravkin-baraldi-orban-2022,aravkin-baraldi-orban-2024} give procedures based on the solution of a nonlinear equation to obtain an element of~\eqref{eq:model-first-order} with the additional constraint \(\|s\|_{\infty} \leq \Delta\), where \(\Delta > 0\), or, equivalently, with the additional term \(\chi(s \mid \Delta \B_\infty)\) in the objective, where \(\B_\infty\) is the \(\ell_\infty\)-norm unit ball and \(\chi\) is the indicator of a set.
They do so for two choices of \(\psi\).
Our results apply directly to both regularizers, and indeed to any regularizer combined with a trust-region constraint.
Because \(\B_2 \subset \B_\infty\), \(\|\skcp\|_2 \leq \Delta\).
Thus, we may use the stopping condition \(\|\hskcp\| \geq \kappa_s \Delta\).

\section{Numerical experiments}%
\label{sec:numerical-experiments}

In this section, we present numerical experiments indicating that exploiting inexact objective values, gradients and proximal operators can reduce computational cost substantially.
We implement \Cref{alg:ir2n} in the Julia language \citep{bezanson-edelman-karpinski-shah-2017} as a modification of the R2N solver \citep{diouane-habiboullah-orban-2024} in \citep{baraldi-leconte-orban-2024}.
To the best of our knowledge, there exists no other proximal-type solver able to handle inexact objective, gradient and proximal evaluations for nonconvex \(h\) with which to perform a comparison.

The implementation of the proximal operator of~\eqref{eq:inexact_prox-lp} and~\eqref{eq:inexact_prox-tvp}, which are both convex, is available from the Julia interface \citep{allaire-orban-montoison-2024} to the \href{https://github.com/albarji/proxTV}{\texttt{proxTV}} library \citep{barbero-sra-2018}.
Both implement iterative methods.
The method for~\eqref{eq:inexact_prox-lp} computes projected quasi-Newton search directions, and performs a backtracking line search to determine the step size.
That for~\eqref{eq:inexact_prox-tvp} alternates between gradient projection into the \(\ell_p\)-norm ball and Frank-Wolfe steps.
After each update, the primal solution is reconstructed from the dual variable, and a new gradient is computed.

Our implementation of the proximal operator of~\eqref{eq:inexact_prox-chi} is based on the Iteratively Reweighted \(\ell_p\)-Ball Projection (IRBP) scheme of \citep{yang-wang-wang-2022}.
At each iteration, IRBP approximates the \(\ell_p\)-``ball'' norm via a weighted linearization of the nonconvex set around the current iterate.
This results in a convex subproblem describing a projection into a weighted \(\ell_1\)-norm ball, which can be solved efficiently \citep{condat-2016}.
A smoothing vector is maintained and adaptively updated to avoid numerical instability and improve convergence.
The nonconvex nature of \(\chi_{p, r}\) implies that there may be non-global minima or saddle points \citep{yang-wang-wang-2022}.
Therefore, the step output by \(\chi_{p, r}\) may not even induce \(\hxikcp \geq 0\).
To the best of our knowledge, there is currently no procedure that is guaranteed to determine a global minimum.
In order to mitigate the issue, we implement a multi-start strategy to increase the odds that \(\hskcp\) be a global solution.
Our strategy is not always successful, but nevertheless often results in acceptable steps.
Part of future work is to find a procedure that identifies a global minimizer.
Our implementation is available from \citep{allaire-orban-2024}.

In each case, inexactness in the proximal operator evaluations is controlled by \(0 < \kappa_s \leq 1\) in \Cref{ass:inexact-decrease}.
For \(\kappa_s \approx 0\), the expectation on the quality of \(\hskcp\) is at its lowest, i.e., \Cref{ass:inexact-decrease} is easiest to satisfy, but~\eqref{eq:stopping-criterion-ir2n} is harder to reach.
Thus, the solver may spend less time inside each (cheap) proximal operator evaluation at the cost of potentially performing more (costly) outer iterations.
On the other hand, when \(\kappa_s \approx 1\), the \(\hskcp\) should be close to an exact solution.
In this case, the solver may spend more time than necessary inside each proximal operator evaluation, which may adversely affect the total solution time.
In our experiments, we vary the value of \(\kappa_s\) to assess the impact of the inexactness on the performance of iR2N.

Step~\ref{stp:sk} in \Cref{alg:ir2n} is performed by a special case of \Cref{alg:ir2n} with \(B = 0\) in which the proximal step computation is the only subproblem.
In effect, that is a variant of the R2 algorithm \citep[Algorithm~6.1]{aravkin-baraldi-orban-2022} extended to the inexact proximal framework.
We refer to this variant as iR2.
Although iR2 is also allowed to perform inexact evaluations of its smooth objective and gradient, we evaluate the quadratic model \(\varphi(s; x_k)\) exactly in our experiments.

Each procedure to solve~\eqref{eq:inexact_prox-lp}--\eqref{eq:inexact_prox-chi} comes with its original stopping condition.
We say that we run iR2N in \emph{exact} mode when we use this original stopping condition, independently of \Cref{ass:inexact-decrease}, and we consider that the resulting proximal operator is then evaluated exactly.
By contrast, we run iR2N in \emph{inexact} mode when the iterations of the proximal operator evaluation are terminated as soon as either (i) \(\|\hskcp\| \geq \kappa_s M_k\), where \(M_k\) is the upper bound on \(\|\skcp\|\) given in~\eqref{eq:bounds-on-lp},~\eqref{eq:bound-on-tvp}, or~\eqref{eq:bound-on-chi}, or (ii) the original stopping condition of the proximal operator evaluation is met.
In proximal operator evaluations, iR2 uses the same value of \(\kappa_s\) as iR2N.

Inequalities~\eqref{eq:link-inexact-measure} suggest using \(\nu_k^{-1}\|\hskcp\| \leq \kappa_s \epsilon\) as stopping criterion in \Cref{alg:ir2n}, since it guarantees that \(\nu_k^{-1} \|\skcp\| \leq \epsilon\).
However, we will see that small values of \(\kappa_s\) yield the best performance but make that stopping condition overly stringent.
In addition, the bound \(M_k\) given in \Cref{lem:bounds-on-lp,lem:bound-on-matrix-regularizer,lem:bound-on-chi} need not be tight, and could indeed be quite loose.
For those reasons, all our experiments use the simple stopping condition
\begin{equation}
  \label{eq:stopping-criterion-ir2n}
  \nu_k^{-1}\|\hskcp\| \leq \epsilon.
\end{equation}

In the next sections, we report the performance of iR2N on problems that use the inexact proximal operators described above.
In \Cref{subsec:bpdn,subsec:matrix_completion,subsec:fitzhugh-nagumo}, both the objective and gradient are assumed to be evaluated exactly, i.e., only subject to the limits of floating-point operations.
In \Cref{sec:inexact-objective-and-gradient}, we consider inexact evaluations of the objective and gradient.
All our tests are performed in double precision on a 2020 MacBook Air with an M1 chip (8-core CPU, 8 GB unified memory).

Because \(f\) in our test problems is based on randomly-generated data, we average the statistics over \(10\) runs.
It is useful to keep in mind that each iR2N and iR2 iteration evaluates a single proximal operator---see Line~\eqref{stp:skcp} of \Cref{alg:ir2n}.
Tables in the next sections use the following headers: ``\(\kappa_s\)'' is the value of the inexactness parameter in \Cref{ass:inexact-decrease}, ``iR2N'' is the average number of outer iterations, ``iR2'' is the average number of inner iterations per outer iteration, ``prox'' is the average number of iterations per proximal operator evaluation, and ``time (s)'' is the average CPU solution time in seconds.

\subsection{Basis pursuit denoising problem (BPDN)}%
\label{subsec:bpdn}

The BPDN problem is stated as
\begin{equation}%
  \label{eq:bpdn_reg}
  \minimize{x \in \R^n} \tfrac{1}{2} \|Ax - b\|_2^2 + \mu \|x\|_p,
\end{equation}
where \(\mu = 10^{-1}\), \(A \in \R^{200 \times 512}\) is random with orthonormal rows, \(b = A \bar{x} + \varepsilon\), \(\bar{x}\) has \(10\) nonzeros, and \(\varepsilon\) is a noise vector from a normal \((0, 1)\) distribution.
We use \(p = 1.1 \) to attempt to recover a sparse solution.
In~\eqref{eq:stopping-criterion-ir2n}, we set \(\epsilon = 10^{-6}\).

\begin{table}[ht]%
  \centering
  \footnotesize
  \caption{Statistics on~\eqref{eq:bpdn_reg} for several values of \(\kappa_s\).}%
  \label{tab:bpdn}
  \begin{tabular}{r r r r r}%
    \hline
    $\kappa_s$        & iR2N             & iR2              & prox             & time (s)         \\
    \hline
    \(1.00\)e\(-07\)  & \(1.61\)e\(+01\) & \(1.21\)e\(+02\) & \(1.02\)e\(+02\) & \(5.03\)e\(+00\) \\
    \(1.00\)e\(-05\)  & \(1.57\)e\(+01\) & \(1.63\)e\(+02\) & \(1.90\)e\(+02\) & \(9.80\)e\(+00\) \\
    \(1.00\)e\(-03\)  & \(1.49\)e\(+01\) & \(1.33\)e\(+01\) & \(4.02\)e\(+02\) & \(1.55\)e\(+01\) \\
    \(1.00\)e\(-02\)  & \(1.49\)e\(+01\) & \(1.78\)e\(+01\) & \(6.02\)e\(+02\) & \(1.77\)e\(+01\) \\
    \(1.00\)e\(-01\)  & \(1.45\)e\(+01\) & \(1.39\)e\(+01\) & \(5.81\)e\(+02\) & \(1.32\)e\(+01\) \\
    \(5.00\)e\(-01\)  & \(1.45\)e\(+01\) & \(1.37\)e\(+01\) & \(5.90\)e\(+02\) & \(1.28\)e\(+01\) \\
    \(9.00\)e\(-01\)  & \(1.45\)e\(+01\) & \(1.39\)e\(+01\) & \(5.80\)e\(+02\) & \(1.25\)e\(+01\) \\
    \(9.90\)e\(-01\)  & \(1.46\)e\(+01\) & \(1.37\)e\(+01\) & \(5.90\)e\(+02\) & \(1.38\)e\(+01\) \\
    \text{exact mode} & \(1.45\)e\(+01\) & \(1.35\)e\(+01\) & \(5.68\)e\(+02\) & \(1.20\)e\(+01\) \\
    \hline
  \end{tabular}
\end{table}

\Cref{tab:bpdn} shows that the average number of iR2N/iR2 iterations decreases globally as \(\kappa_s\) increases.
The proximal operator iterations increase as \(\kappa_s\) increases, as expected.
For small values of \(\kappa_s\), inexact mode yields a substantial reduction in the number of proximal iterations and solution time compared with exact mode at the expense of a modest increase in outer iterations.
For large values of \(\kappa_s\) the behavior of iR2N is close to that of exact mode.

\Cref{fig:bpdn-ir2n} shows that the solutions produced in exact and inexact mode are essentially identical, and that both recover the sparse support of \(\bar{x}\).

\begin{figure}[ht]
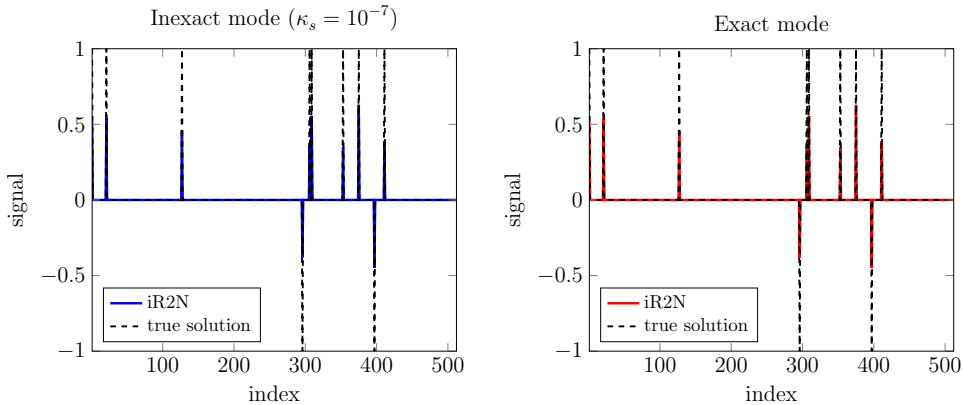
%
  \centering
  \resizebox{\linewidth}{!}{%
    \includetikzgraphics{bpdn-exact-vs-inexact}
  }
  \caption{Components of the solution of~\eqref{eq:bpdn_reg} found by iR2N and of \(\bar{x}\).}%
  \label{fig:bpdn-ir2n}
\end{figure}

\subsection{Matrix completion problem}%
\label{subsec:matrix_completion}

The problem is stated as
\begin{equation}%
  \label{eq:matrix_completion_reg}
  \minimize{x \in \R^n} \tfrac{1}{2} \left\| P(x - A)\right\|_F^2 + \mu \text{TV}_p(x),
\end{equation}
where \(\mu = 10^{-1}\), \(p = 1.1\) and \(A \in \R^{10 \times 12}\) is a fixed matrix representing an image and the operator \(P\) only retains a subset of pixels.
In~\eqref{eq:stopping-criterion-ir2n}, \(\epsilon = 10^{-3}\).

\Cref{tab:matrix_completion} gathers our results on~\eqref{eq:matrix_completion_reg}.
The benefits of choosing \(\kappa_s\) small are similar to those in \Cref{tab:bpdn}.
\Cref{fig:matcomp-ir2n} shows that the reconstruction error with the solutions of exact and inexact mode are close, as is the discrepancy between the two solutions.

\begin{table}[ht]%
  \centering
  \footnotesize
  \caption{Statistics on~\eqref{eq:matrix_completion_reg} for several values of \(\kappa_s\).}%
  \label{tab:matrix_completion}
  \begin{tabular}{r r r r r}%
    \hline
    $\kappa_s$        & iR2N             & iR2              & prox             & time (s)         \\
    \hline
    \(1.00\)e\(-07\)  & \(3.69\)e\(+01\) & \(3.41\)e\(+02\) & \(5.88\)e\(+02\) & \(9.46\)e\(+01\) \\
    \(1.00\)e\(-05\)  & \(3.72\)e\(+01\) & \(3.03\)e\(+02\) & \(8.71\)e\(+02\) & \(1.42\)e\(+02\) \\
    \(1.00\)e\(-03\)  & \(3.69\)e\(+01\) & \(2.09\)e\(+02\) & \(3.76\)e\(+03\) & \(3.54\)e\(+02\) \\
    \(1.00\)e\(-02\)  & \(3.77\)e\(+01\) & \(2.12\)e\(+02\) & \(4.06\)e\(+03\) & \(3.73\)e\(+02\) \\
    \(1.00\)e\(-01\)  & \(3.41\)e\(+01\) & \(1.90\)e\(+02\) & \(4.37\)e\(+03\) & \(3.25\)e\(+02\) \\
    \(5.00\)e\(-01\)  & \(3.56\)e\(+01\) & \(2.19\)e\(+02\) & \(4.31\)e\(+03\) & \(3.54\)e\(+02\) \\
    \(9.00\)e\(-01\)  & \(3.77\)e\(+01\) & \(1.81\)e\(+02\) & \(4.49\)e\(+03\) & \(3.57\)e\(+02\) \\
    \(9.90\)e\(-01\)  & \(3.55\)e\(+01\) & \(2.01\)e\(+02\) & \(4.27\)e\(+03\) & \(3.54\)e\(+02\) \\
    \text{exact mode} & \(3.18\)e\(+01\) & \(1.67\)e\(+02\) & \(4.49\)e\(+03\) & \(3.36\)e\(+02\) \\
    \hline
  \end{tabular}
\end{table}

\begin{figure}[ht]
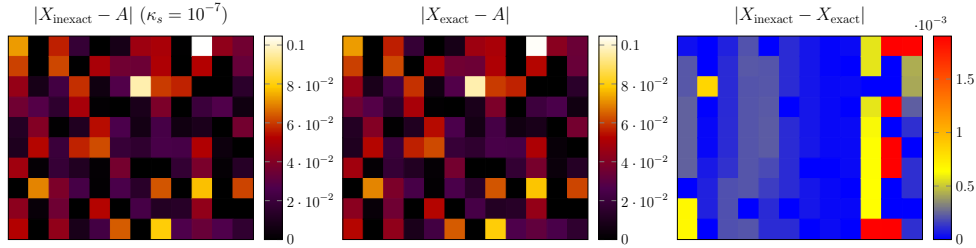
%
  \centering
  \resizebox{\linewidth}{!}{%
    \includetikzgraphics{demo-matcomp-comparison}
  }
  \caption{Left: Heatmap of the difference between the solution \(X\) found by iR2N in inexact and exact mode, and \(A\).
    Right: Difference between the two solutions.
    The values masked by \(P\) are set to zero and shown in black.}%
  \label{fig:matcomp-ir2n}
\end{figure}

\subsection{Fitzhugh-Nagumo inverse problem}%
\label{subsec:fitzhugh-nagumo}

The FitzHugh–Nagumo system is a simplified representation of a neuron's action potential modeled by the system of differential equations
\begin{equation}%
  \label{eq:fitzhugh_nagumo}
  \dot{v}(t) = x_2^{-1} (v(t) - \tfrac{1}{3} v(t)^3 - w(t) + x_1), \quad
  \dot{w}(t) = x_2 (x_3 v(t) - x_4 w(t) + x_5).
\end{equation}
We use initial conditions \(v(0) = 2\) and \(w(0) = 0\), and generate data \(\bar{v}(x), \bar{w}(x)\) by solving~\eqref{eq:fitzhugh_nagumo} with \(\bar{x} = (0, 0.2, 1, 0, 0)\), which corresponds to the Van der Pol oscillator, to which we add random noise.
We then aim to recover \(\bar{x}\) by minimizing the misfit while encouraging a sparse solution:
\begin{equation}%
  \label{eq:fitzhugh_nagumo_reg}
  \min_{x \in \R^5} \tfrac{1}{2} \left\| F(x) \right\|_2^2 + \chi_{p, r}(x),
\end{equation}
where \(p = 0.5\), \(r = 2\), \(F : \R^5 \to \R^{2n+2}\), \(F(x) \defeq (v(x) - \bar{v}(x), w(x) - \bar{w}(x))\), and \(v(x) = (v_1(x), \ldots, v_{n+1}(x))\) and \(w(x) = (w_1(x), \ldots, w_{n+1}(x))\) are sampled values of \(V\) and \(W\) at \(n+1\) discretization points.
We set \(\epsilon = 10^{-5}\) in~\eqref{eq:stopping-criterion-ir2n}.
\Cref{tab:fh} reports our results.

\begin{table}[ht]%
  \centering
  \footnotesize
  \caption{Statistics on~\eqref{eq:fitzhugh_nagumo_reg} for \(p = \tfrac{1}{2}\) and \(r = 2\) with several values of \(\kappa_s\).}%
  \label{tab:fh}
  \begin{tabular}{r r r r r}%
    \hline
    $\kappa_s$        & iR2N             & iR2              & prox             & time (s)         \\
    \hline
    \(1.00\)e\(-07\)  & \(5.14\)e\(+02\) & \(4.90\)e\(+02\) & \(3.51\)e\(-01\) & \(5.28\)e\(+00\) \\
    \(1.00\)e\(-05\)  & \(5.72\)e\(+02\) & \(4.64\)e\(+02\) & \(4.62\)e\(-01\) & \(5.21\)e\(+00\) \\
    \(1.00\)e\(-03\)  & \(6.31\)e\(+02\) & \(5.47\)e\(+02\) & \(5.96\)e\(-01\) & \(5.56\)e\(+00\) \\
    \(1.00\)e\(-02\)  & \(5.71\)e\(+02\) & \(4.81\)e\(+02\) & \(6.22\)e\(-01\) & \(5.17\)e\(+00\) \\
    \(1.00\)e\(-01\)  & \(4.95\)e\(+02\) & \(4.89\)e\(+02\) & \(4.11\)e\(-01\) & \(5.85\)e\(+00\) \\
    \(5.00\)e\(-01\)  & \(4.90\)e\(+02\) & \(4.59\)e\(+02\) & \(1.94\)e\(+00\) & \(6.42\)e\(+00\) \\
    \(9.00\)e\(-01\)  & \(5.12\)e\(+02\) & \(4.98\)e\(+02\) & \(2.06\)e\(+00\) & \(6.53\)e\(+00\) \\
    \(9.90\)e\(-01\)  & \(5.24\)e\(+02\) & \(5.09\)e\(+02\) & \(1.91\)e\(+00\) & \(6.84\)e\(+00\) \\
    \text{exact mode} & \(4.92\)e\(+02\) & \(5.03\)e\(+02\) & \(3.92\)e\(+01\) & \(6.88\)e\(+00\) \\
    \hline
  \end{tabular}
\end{table}

The small number of iterations per proximal call arises from the fact that \(\chi_{p,r}\) is an indicator; the projection of a point that already belongs to the set requires zero iterations.
The value of \(\kappa_s\) has little effect on the number of iR2N/iR2 iterations.
As in \Cref{subsec:bpdn,subsec:matrix_completion}, inexact mode yields a reduction in computational cost, though more modest because the smooth objective and its gradient are costlier in~\eqref{eq:fitzhugh_nagumo_reg} than in~\eqref{eq:bpdn_reg} or~\eqref{eq:matrix_completion_reg}.
Thus, the reduction in proximal evaluations must outweigh the increase in outer iterations.
\Cref{tab:fitzhugh_nagumo_sol} gives the approximate solution identified by the exact and inexact variants, and the final value of the smooth objective.
Although both exact and inexact mode recover a solution that has one more nonzero than \(\bar{x}\), the final smooth objective values are close to that at \(\bar{x}\).
\Cref{fig:fitzhugh_nagumo-ir2n} plots the simulation of~\eqref{eq:fitzhugh_nagumo} with parameters found by iR2N with \(\kappa_s = 1.0\)e\(-07\) when solving~\eqref{eq:fitzhugh_nagumo_reg}.
The solutions with exact and inexact mode are indistinguishable.


\begin{table}[ht]%
  \centering
  \footnotesize
  \caption{Approximate solutions of~\eqref{eq:fitzhugh_nagumo_reg} found by the exact and inexact variants with \(\kappa_s = 1.0\)e\(-07\).
    The last column shows the smooth objective value at the solution.
  }%
  \label{tab:fitzhugh_nagumo_sol}
  \begin{tabular}{l|rrrrr|r}
    \hline
            & \multicolumn{5}{c|}{\(x\)} & \(\tfrac{1}{2} \|F(x)\|^2\)                                                                              \\
    \hline
    True    & \(0.00\)e\(+00\)           & \(2.00\)e\(-01\)            & \(1.00\)e\(+00\) & \(0.00\)e\(+00\)  & \(0.00\)e\(+00\) & \(8.82\)e\(-01\) \\
    Inexact & \(0.00\)e\(+00\)           & \(2.00\)e\(-01\)            & \(9.98\)e\(-01\) & \(-1.00\)e\(-02\) & \(0.00\)e\(+00\) & \(8.96\)e\(-01\) \\
    Exact   & \(0.00\)e\(+00\)           & \(2.00\)e\(-01\)            & \(9.98\)e\(-01\) & \(-1.00\)e\(-02\) & \(0.00\)e\(+00\) & \(8.96\)e\(-01\) \\
    \hline
  \end{tabular}
\end{table}

\begin{figure}[ht]
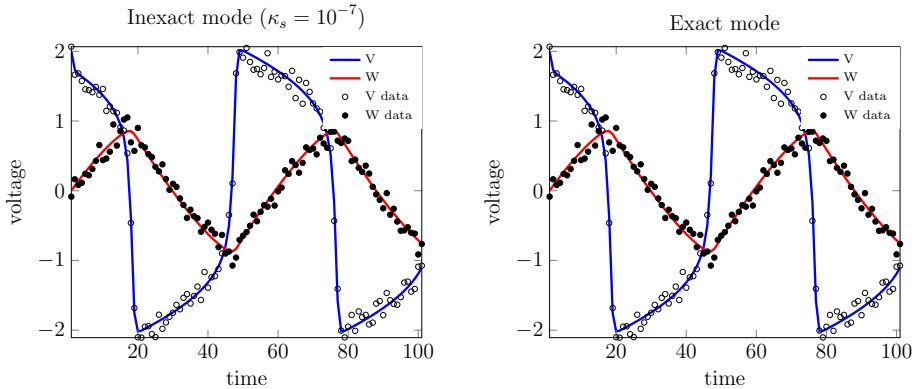
%
  \centering
  \resizebox{.95\linewidth}{!}{%
    \includetikzgraphics{fh-ir2n-exact-inexact}
  }
  \caption{Simulation of~\eqref{eq:fitzhugh_nagumo} with solutions of~\eqref{eq:fitzhugh_nagumo_reg} found by iR2N.}%
  \label{fig:fitzhugh_nagumo-ir2n}
\end{figure}

\subsection{Inexact objective and gradient evaluations}%
\label{sec:inexact-objective-and-gradient}

We now consider inexact evaluations of the smooth objective and its gradient.
In~\eqref{eq:fitzhugh_nagumo_reg}, each evaluation of \(F\) involves solving an ODE system numerically, which inherently depends on a stopping tolerance that introduces an approximation error.
We use the \citet{verner-2010} \(9/8\) optimal Runge-Kutta method as implemented in \citep{differentialequations.jl-2017}.
In our implementation of \(F\), the accuracy of the ODE solve can be adjusted via a parameter \texttt{prec} \(> 0\) that sets the absolute and relative stopping tolerances.
The gradient is computed via automatic differentiation, and hence, its accuracy also depends on \texttt{prec}.
Decreasing this tolerance improves the accuracy of the objective and gradient but increases the computational cost.
The results of \Cref{subsec:fitzhugh-nagumo} used \(\texttt{prec} = 10^{-14}\) as the reference ``exact'' objective and gradient evaluations.

Because \Cref{ass:bound-inexact-f-nf} may not be easily verifiable in practice, we propose a heuristic inspired from trust-region methods for derivative-free optimization \citep[chapter 10]{conn-scheinberg-vicente-2009}, that consists in adapting the accuracy based on the progress of the algorithm.
More precisely, we increase the accuracy on unsuccessful iterations, i.e., \(\rho_k < \eta_1\) in \Cref{alg:ir2n}.
At iteration \(k\), we set \texttt{prec} to
\begin{equation}
  \label{eq:prec_k_failure}
  \texttt{prec}(k) \defeq \max(10^{-3} \, \exp(\log(10^{-14} / 10^{-3}) \, n_F/N), \, 10^{-14}),
\end{equation}
where \(N\) is a preset maximum number of unsuccessful iterations after which \texttt{prec} \(= 10^{-14}\) is always used, and \(n_F\) counts the number of unsuccessful iterations.
Small values of \(N\) lead to a rapid increase in accuracy, whereas larger ones maintain low-accuracy evaluations over more iterations.
Though~\eqref{eq:prec_k_failure} may not guarantee \Cref{ass:bound-inexact-f-nf} at every iteration, the objective and gradient accuracy improves as the algorithm progresses, as required by the assumption.

We focus on~\eqref{eq:fitzhugh_nagumo_reg} with the setting of \Cref{subsec:fitzhugh-nagumo} and we use~\eqref{eq:prec_k_failure} for inexact objective and gradient.
We vary the value of \(N\) with fixed \(\kappa_s = 10^{-7}\) in \Cref{tab:fitzhugh_nagumo_inexact_exact_prox}.

\begin{table}[ht]%
  \centering
  \footnotesize
  \caption{Iterations and time on~\eqref{eq:fitzhugh_nagumo_reg} with inexact objective and gradient evaluations.}%
  \label{tab:fitzhugh_nagumo_inexact_exact_prox}
  \begin{tabular}{r r r r r r}
    \hline
    \(N\)       & fail rate & iter iR2N        & iter iR2         & prox             & time (s)         \\
    \hline
    exact \(F\) & 0\%       & \(5.14\)e\(+02\) & \(4.90\)e\(+02\) & \(3.51\)e\(-01\) & \(5.28\)e\(+00\) \\
    \(20\)      & 0\%       & \(5.66\)e\(+02\) & \(5.10\)e\(+02\) & \(4.55\)e\(-01\) & \(5.16\)e\(+00\) \\
    \(50\)      & 20\%      & \(6.36\)e\(+02\) & \(5.07\)e\(+02\) & \(3.77\)e\(-01\) & \(4.31\)e\(+00\) \\
    \(100\)     & 30\%      & \(6.31\)e\(+02\) & \(5.08\)e\(+02\) & \(3.46\)e\(-01\) & \(3.27\)e\(+00\) \\
    \(200\)     & 80\%      & \(6.67\)e\(+02\) & \(5.47\)e\(+02\) & \(3.69\)e\(-01\) & \(2.46\)e\(+00\) \\
    \hline
  \end{tabular}
\end{table}

The first line of \Cref{tab:fitzhugh_nagumo_inexact_exact_prox} reports the number of iterations and the solution time obtained with ``exact'' objective and gradient.
Lines~2--5 use~\eqref{eq:prec_k_failure} for several values of \(N\).
As \(N\) increases, iR2N spends a larger fraction of its iterations in a \emph{low}-precision regime, making it increasingly likely that \Cref{ass:bound-inexact-f-nf} is violated.
When iR2N operates with insufficient accuracy for too long, the algorithm may eventually stall, cease to make progress, and reach the maximum number of allowed iterations.
The second column of \Cref{tab:fitzhugh_nagumo_inexact_exact_prox} reports the proportion of such failed runs over ten trials.
Importantly, the iteration and timing statistics shown in \Cref{tab:fitzhugh_nagumo_inexact_exact_prox} correspond \emph{only} to the successful runs.
The failure rate increases with \(N\), and for \(N = 200\) few runs succeed.
Moderate values of \(N\) yield significant benefits in terms of solution time.


In \Cref{tab:fitzhugh_nagumo_inexact}, we report the performance of \Cref{alg:ir2n} using inexact objective, gradient and proximal operator evaluations following rule~\eqref{eq:prec_k_failure} on~\eqref{eq:fitzhugh_nagumo_reg} with \(N = 100\).
The number of iR2N, iR2 and proximal iterations is globally unaffected by inexact evaluations, but the latter yield significant savings in terms of solution time.

\begin{table}[ht]%
  \centering
  \footnotesize
  \caption{Statistics on~\eqref{eq:fitzhugh_nagumo_reg} with increasing accuracy given by~\eqref{eq:prec_k_failure} with \(N = 100\) and several values of \(\kappa_s\).
    Each entry reports the multiplicative gain or loss compared to the reference values in \Cref{tab:fh}.
    A value smaller than \(1\) indicates a gain.}%
  \label{tab:fitzhugh_nagumo_inexact}
  \begin{tabular}{r r r r r}%
    \hline
    $\kappa_s$            & iR2N             & iR2              & prox             & time (s)         \\
    \hline
    \(1.00\)e\(-07\)      & \(1.23\)e\(+00\) & \(1.04\)e\(+00\) & \(9.90\)e\(-01\) & \(6.20\)e\(-01\) \\
    \(1.00\)e\(-05\)      & \(1.08\)e\(+00\) & \(1.02\)e\(+00\) & \(1.38\)e\(+00\) & \(4.90\)e\(-01\) \\
    \(1.00\)e\(-03\)      & \(8.40\)e\(-01\) & \(7.70\)e\(-01\) & \(5.50\)e\(-01\) & \(2.70\)e\(-01\) \\
    \(1.00\)e\(-02\)      & \(1.00\)e\(+00\) & \(1.00\)e\(+00\) & \(1.10\)e\(+00\) & \(3.60\)e\(-01\) \\
    \(1.00\)e\(-01\)      & \(1.11\)e\(+00\) & \(9.60\)e\(-01\) & \(5.20\)e\(-01\) & \(3.00\)e\(-01\) \\
    \(5.00\)e\(-01\)      & \(9.90\)e\(-01\) & \(9.20\)e\(-01\) & \(1.19\)e\(+00\) & \(2.50\)e\(-01\) \\
    \(9.00\)e\(-01\)      & \(1.03\)e\(+00\) & \(8.80\)e\(-01\) & \(1.17\)e\(+00\) & \(3.00\)e\(-01\) \\
    \(9.90\)e\(-01\)      & \(9.40\)e\(-01\) & \(8.30\)e\(-01\) & \(1.36\)e\(+00\) & \(2.50\)e\(-01\) \\
    \hline
    \text{average factor} & \(1.03\)e\(+00\) & \(9.30\)e\(-01\) & \(1.03\)e\(+00\) & \(3.60\)e\(-01\) \\
    \hline
  \end{tabular}
\end{table}

\section{Discussion}%
\label{sec:discussion}

Method iR2N subsumes R2N \citep{diouane-habiboullah-orban-2024} by allowing inexact evaluations of the objective, its gradient, and the proximal operator.
Under usual global convergence conditions, we showed that inexact evaluations and proximal operators do not deteriorate asymptotic complexity compared to methods that use exact evaluations.
Our assumptions on the inexactness of \(f\) and \(\nabla f\) are standard.

\Cref{ass:inexact-decrease} on the inexact evaluation of proximal operators differs in nature from Definitions~(ii) and~(iii) of \citep{salzo-villa-2012}.
Their Definition~(i), also used in \citep{rockafellar-1976}, can be written \(\|\hskcp - \skcp\| \leq \epsilon_k\) for at least one \(\skcp\), where \(\{\epsilon_k\}\) is positive and summable.
It is equivalent to \(\|\skcp\| - \epsilon_k \leq \|\hskcp\| \leq \|\skcp\| + \epsilon_k\), which is strictly stronger than \Cref{ass:inexact-decrease} in that we only require one of the inequalities.
Moreover, we use the specific value \(\epsilon_k = (1 - \kappa_s) \|\skcp\|\), which need not be summable.
Indeed, by the same reasoning as in the proof of \Cref{lem:inf-succ}, for any successful iteration \(k\), there exists a Cauchy step \(\skcp\) such that
\begin{align*}
  (f + h)(x_k) - (f + h)(x_k + s_k) & \geq \tfrac{1}{2} \eta_1 (1 - \theta_1) \nu_k^{-1} \|\hskcp\|^2
  \\
                                    & \geq \tfrac{1}{2} \eta_1 (1 - \theta_1) \nu_{\max}^{-1} \|\hskcp\|^2
  \\
                                    & \geq \tfrac{1}{2} \eta_1 (1 - \theta_1) \nu_{\max}^{-1} \kappa_s^2 \|\skcp\|^2.
\end{align*}
Therefore, if we sum those inequalities over the set \(\mathcal{S}\) of all successful iterations and use \Cref{ass:fh-lower-bound}, we obtain
\[
  (f + h)(x_0) - (f + h)_{\text{low}} \geq \tfrac{1}{2} \eta_1 (1 - \theta_1) \nu_{\max}^{-1} \kappa_s^2 \sum_{k \in \mathcal{S}} \|\skcp\|^2.
\]
A similar inequality holds for \(\hskcp\).
Thus, both \(\{\|\hskcp\|\}\) and \(\{\|\skcp\|\}\) are square summable.
However, showing that they are summable appears to require the stronger Kurdyka-\L{}ojasiewicz assumption \citep[Theorem~\(1\)]{bolte-sabach-teboulle-2014}, which is not used in our analysis.

iR2N naturally generalizes the special cases R2 \citep{aravkin-baraldi-orban-2022} with \(B(x) = 0\), R2DH \citep{diouane-habiboullah-orban-2024} with \(B(x)\) diagonal, and LM \citep{aravkin-baraldi-orban-2024} when \(f\) is a squared residual norm and \(B(x) = J(x) J(x)^T\), where \(J(x)\) is the residual Jacobian.
It stands to reason that the same mechanisms can be used to extend the trust-region variants TR \citep{aravkin-baraldi-orban-2022}, TRDH \citep{leconte-orban-2025}, and LMTR \citep{aravkin-baraldi-orban-2024} to inexact evaluations and proximal operators with minimal modifications.

Numerical experiments confirm that iR2N provides substantial flexibility in contexts where exact evaluations are expensive or unavailable, and demonstrate that controlled inexactness can be leveraged to reduce computational cost without compromising convergence behavior.

Future work will focus on allowing inexact evaluations of the quadratic model~\eqref{eq:phi-second-order}, particularly regarding \(B_k\), which itself may be computed inexactly---for instance, when represented in reduced numerical precision or when linear systems involving \(B_k\) are solved approximately.





\small
\bibliographystyle{abbrvnat}
\bibliography{abbrv,ir2n}
\normalsize

\newpage

\hypertarget{contents}{}  
\tableofcontents

\end{document}